\documentclass[reqno,10pt]{amsart}
\usepackage[english]{babel}
\usepackage{amsmath}

\usepackage{amssymb}
\usepackage[scr=boondoxo]{mathalfa}
\usepackage{graphicx}
\usepackage{setspace}
\usepackage{epstopdf}

\usepackage{subfigure}
\usepackage{color}

\newcommand{\bpr}{\begin{trivlist} \item[]{\bf Proof. }}
\newcommand{\epr}{\hspace*{\fill} $\qed$\end{trivlist}}

\newcommand{\be}{\begin{eqnarray}}
\newcommand{\ee}{\end{eqnarray}}
\newcommand{\ba}{\begin{align}}
\newcommand{\ea}{\end{align}}
\newcommand{\bi}{\begin{itemize}}
\newcommand{\ei}{\end{itemize}}

\newcommand{\secref}[1]{Section~\ref{sec:#1}}

\newcommand{\seclab}[1]{\label{sec:#1}}
\newcommand{\eqlab}[1]{\label{eq:#1}}
\renewcommand{\eqref}[1]{(\ref{eq:#1})}

\newcommand{\figref}[1]{Fig.~\ref{fig:#1}}
\newcommand{\figlab}[1]{\label{fig:#1}}

\newcommand{\lemmaref}[1]{Lemma~\ref{lemma:#1}}
\newcommand{\lemmalab}[1]{\label{lemma:#1}}
\newcommand{\remref}[1]{Remark~\ref{remark:#1}}
\newcommand{\remlab}[1]{\label{remark:#1}}
\newcommand{\corref}[1]{Corollary~\ref{cor:#1}}
\newcommand{\corlab}[1]{\label{cor:#1}}
\newcommand{\thmref}[1]{Theorem~\ref{theorem:#1}}
\newcommand{\thmlab}[1]{\label{theorem:#1}}

\newcommand{\appref}[1]{Appendix~\ref{app:#1}}

\newcommand{\applab}[1]{\label{app:#1}}

 \newcommand\qs[1]{{\textcolor{black}{#1}}}
 \newcommand\qss[1]{{\textcolor{black}{#1}}}

\newtheorem{theorem}{Theorem}[section]

\newtheorem{lemma}[theorem]{Lemma}
\newtheorem{cor}[theorem]{Corollary}

\newtheorem{remark}[theorem]{Remark}

\numberwithin{equation}{section}

\begin{document}

\title{The regularized visible fold revisited }

\author {K. Uldall Kristiansen} 
\date\today
\maketitle

\vspace* {-2em}
\begin{center}
\begin{tabular}{c}
Department of Applied Mathematics and Computer Science, \\
Technical University of Denmark, \\
2800 Kgs. Lyngby, \\
DK
\end{tabular}
\end{center}

 \begin{abstract}
 The planar visible fold is a simple singularity in piecewise smooth systems. In this paper, we consider singularly perturbed systems that limit to this piecewise smooth bifurcation as the singular perturbation parameter $\epsilon\rightarrow 0$. Alternatively, these singularly perturbed systems can be thought of as regularizations of their piecewise counterparts. The main contribution of the paper is to demonstrate the use of consecutive blowup transformations in this setting, allowing us to obtain detailed information about a transition map near the fold under very general assumptions. We apply this information to prove, \qs{for the first time}, the existence of a locally unique saddle-node bifurcation in the case where a limit cycle, in the singular limit $\epsilon\rightarrow 0$, grazes the discontinuity set. We apply this result to a mass-spring system on a moving belt described by a Stribeck-type friction law. 
 \end{abstract}
\section{Introduction}

Piecewise smooth (PWS) differential equations appear in many applications, including problems in mechanics (impact, friction, backlash, free-play, gears, rocking blocks), see also \secref{model} below, electronics (switches and diodes, DC/DC converters, $\Sigma-\Delta$ modulators), control engineering (sliding mode control, digital control, optimal control), oceanography (global circulation models), economics (duopolies) and biology (genetic regulatory networks): see \cite{Bernardo08, MakarenkovLamb12} for further references. \qs{Much of the mathematical study of PWS systems began with the work of Filippov \cite{filippov1988differential} and Utkin \cite{utkin1992a}, the latter with a strong focus on applications in control theory}. However, PWS models do pose mathematical difficulties because they do not in general define a (classical) dynamical system. In particular, forward {uniqueness} of solutions cannot always be guaranteed; a prominent example of this is the two-fold in $\mathbb R^3$, see \cite{desroches_canards_2011}. 

Frequently, PWS systems are idealisations of smooth systems with abrupt transitions. It is therefore perhaps natural to view a PWS system as a singular limit of a smooth regularized system. This viewpoint has been adopted by many authors, see e.g. \cite{Sotomayor96,Buzzi06,guglielmi2015a,Llibre07,llibre2009a,kaklamanos2019a,bonet-rev2016a,kristiansen2017a,kristiansen2018a,krihog2,krihog}, and is useful for resolving the ambiguities associated with PWS systems. 
In \cite{kristiansen2018a}, for example, the authors showed that the regularization of the visible-invisible two-fold in $\mathbb R^3$, a PWS singularity producing a loss of uniqueness, possesses a forward orbit $U$ that is distinguished amongst all the possible forward orbits as $\epsilon\rightarrow 0$, see \cite[Theorem 1]{kristiansen2018a} for details. Although this result was only given for one particular regularization function ($\arctan$), the authors acknowledged that the results could be extended to other functions (including ones like those in (A1) and (A2) below) without essential changes to either their result or their approach. To obtain this result, the authors applied an adaptation of the blowup method, pioneered by Dumortier and Roussarie \cite{dumortier_1996} to deal with fully nonhyperbolic singularities.

\qss{
On the other hand, the aim of \cite{Sotomayor96}, as one of the first papers on regularization,  was to develop a systematic study of PWS singularities and initiate the Peixoto's program about structural stability of PWS systems. In contrast, the present paper is less about regularization and more about continuing a relative new program for the analysis of smooth system that approach nonsmooth ones, as they appear in applications, using blowup techniques in a framework known from from Geometric Singular Perturbation Theory (GSPT) \cite{krupa_extending_2001,kuehn2015,Gucwa2009783}. See \cite{peterTalk,kosiuk2015a,kukszm,kristiansen2019d} for recent research in this direction. By following this approach, we consider -- in a very general framework -- smooth planar systems limiting as $\epsilon\rightarrow 0$  to PWS systems having visible fold singularity. Using blowup, we obtain a detailed and uniform description of a transition map near this fold. This allows us to prove, for the first time, the existence of a locally unique saddle-node bifurcation in the case where a family of limit cycles grazes the discontinuity set in a PWS visible fold.}

\subsection{Setting}\seclab{setting}
We consider planar systems of the following form
\begin{align}
\dot z &=Z(z,\phi(y\epsilon^{-1},\epsilon),\alpha),\eqlab{ztf}
\end{align}
where $z=(x,y)\in \mathbb R^2$ and $\phi:\mathbb R\times [0,\epsilon_0]\rightarrow \mathbb R$. Moreover, $\epsilon\in (0,\epsilon_0]$ and $\alpha \in I\subset \mathbb R$ are parameters and $Z:\mathbb R^2\times \mathbb R\times I\rightarrow \mathbb R^2$ is smooth in all arguments. 

Specifically, we will assume that:
\begin{itemize}
 \item[(A0)] $p\mapsto Z(z,p,\alpha)$ is affine:
%
%
\begin{align}
 Z(z,p,\alpha)=p Z_+(z,\alpha)+(1-p)Z_-(z,\alpha),\eqlab{fA}
\end{align}
with $Z_\pm:\mathbb R^2\times I\rightarrow \mathbb R^2$ each smooth. 
\end{itemize}
Regarding the functions $\phi$ we suppose the following:
\begin{itemize}
 \item[(A1)] $\phi:\mathbb R\times [0,\epsilon_0]\rightarrow \mathbb R$ is a smooth ``regularization function'' satisfying:
\begin{align}
 \phi'_{s}(s,\epsilon):= \frac{\partial \phi(s,\epsilon)}{\partial s}>0,\eqlab{monone}
\end{align}
for all $s\in \mathbb R,\,\epsilon\in [0,\epsilon_0]$ and 
\begin{align}
 \phi(s,\epsilon) \rightarrow \left\{\begin{array}{cc} 1 & \text{for $s\rightarrow \infty$}\\
                                 0 &  \text{for $s\rightarrow -\infty$}
                                \end{array}\right.,\eqlab{phiLimit}
\end{align}
for each $\epsilon\in [0,\epsilon_0]$.
\end{itemize}
\qs{Assumptions (A0) and (A1), specifically \eqref{phiLimit}, imply that the $\epsilon\rightarrow 0$ limit of \eqref{ztf} is well-defined pointwise for $y\ne 0$, the limit being the PWS system:
\begin{align}
 \dot z &= \left\{\begin{matrix}
                                           Z_+(z,\alpha)&\text{for}\,y>0,\\
                                           Z_-(z,\alpha)&\text{for}\,y<0.
                                          \end{matrix}\right.\eqlab{ztfpws}
\end{align}
In this sense, \eqref{ztf} is ``singularly perturbed'', but obviously in a different way to slow-fast systems where
\begin{align*}
 \epsilon \dot x&=f(x,y,\epsilon),\\
 \dot y&=g(x,y,\epsilon),
\end{align*}
for $0<\epsilon\ll 1$.
Such systems have successfully been studied during the past decades by GSPT. This theory provides a general framework or toolbox for dealing with singular perturbations (in the dissipative setting) using invariant manifolds. It consists of various theories and methods, most notably (a) Fenichel's theory \cite{fen1,fen2,fen3}, see also \cite{jones_1995,kuehn2015}, for the perturbation of compact and normally hyperbolic critical submanifolds of $C=\{(x,y)\vert \,f(x,y,0)=0\}$ for $\epsilon=0$, (b) the blowup method  \cite{dumortier_1996,krupa_extending_2001}, for dealing with loss of hyperbolicity of $C$, and finally (c) the Exchange Lemma \cite{schecter2008a}. } 

\qs{Following (A0), see \eqref{ztfpws}, systems of the form \eqref{ztf} can be viewed as regularizations of PWS systems. In this context, the regularization functions $\phi$ are always assumed to be independent of $\epsilon$, see \cite{Buzzi06,guglielmi2015a,Llibre07,llibre2009a,kaklamanos2019a,bonet-rev2016a,kristiansen2017a,kristiansen2018a,krihog2,krihog}.  In this paper, we include this dependency in (A1) for full generality. This makes our results directly applicable to switch-like functions that naturally appear in applications. For example, the Goldbeter-Koshland function \cite{goldbeter1981a}:
\begin{align}
\phi(s,\epsilon) = {\frac {2+\epsilon\,\sqrt {4+s^2+2\,\epsilon\,s^2+4\,\epsilon\,{s}+ {\epsilon}^{2}{{s
}}^{2}}+2\,\epsilon+\epsilon\,{s}+{\epsilon}^{2}{s}}{ 
 \left(2-{s}+ \epsilon\,{s}+\sqrt {4+s^2+2\,
\epsilon\,s^2+4\,\epsilon\,{s}+{\epsilon}^{2}s^2} \right) \left(1+ \epsilon\right) }}.\eqlab{gkfunc}
\end{align}
appears as a steady-state solution for a two-state biological system. It therefore occurs naturally in QSS approximations and for $\epsilon$ small -- which, in the biological context, is given in terms of rate constants, -- it is switch-like. In fact, I have here (without loss of generality) normalised the function appropriately such that \eqref{gkfunc} satisfies \eqref{phiLimit}.
Notice that 
                        \begin{align}
                         \phi(s,0) = \frac{2}{2-s + \sqrt{s^2+4}},\eqlab{gkfunc0}
                        \end{align}
                        and with some extra work
                        \begin{align}
                         \phi(s,\epsilon) = \begin{cases}
                                             1-\frac{s^{-1}}{1+\epsilon} +\mathcal O(s^{-2}),\quad \text{for}\,\,s\rightarrow \infty,\\
                                             \frac{1+2\epsilon}{(1+\epsilon)^{2}} (-s)^{-1} +\mathcal O((-s)^{-2}),\quad \text{for}\,\,s\rightarrow -\infty.
                                            \end{cases}\eqlab{gklimit}
                        \end{align}}
                        
                        \qss{Another reason for considering systems of the form \eqref{ztf}, satisfying the general condition (A1), is that such systems also appear upon certain normalization of systems that are unbounded as $\epsilon \rightarrow 0$. I believe Peter Szmolyan \cite{peterTalk} was the first to promote this connection. For example, singular exponential nonlinearities like $e^{y\epsilon^{-1}}$ with $0<\epsilon\ll 1$ appear in many different areas (see e.g. the Ebers-Moll model of an NPN transistor \cite{ebers1954a} and the Arrhenius law in chemical kinetics \cite{laidler1987a}) of mathematical modelling. Such terms, being unbounded for $y>0$ and $\epsilon\rightarrow 0$, can be ``tamed'' by a normalization through division of the right hand side by a quantity $1+e^{y\epsilon^{-1}}$. This corresponds to a nonlinear transformation of time and produces -- under general assumptions -- systems of the form \eqref{ztf}, satisfying \eqref{fA}, with $\phi$ of the following form:
\begin{align}
\phi(s,\epsilon) = \frac{e^s}{1+e^s}.\eqlab{phicor}
\end{align}
For further details, see the recent preprint \cite{kristiansen2019d}.
}
                        

 \qs{On the other hand, the references \cite{Buzzi06,Llibre07,llibre2009a,bonet-rev2016a,panazzolo2017a}, following \cite{Sotomayor96}, also study a special class, called Sotomayor-Teixera regularization functions, consisting of non-analytic functions $\psi(s)$
 of the following form:
\begin{align*}
 \psi'(s)>0,\quad \text{for all}\quad s\in (-1,1),
\end{align*}
and 
\begin{align}
 \psi(s)&=0,\quad \text{for all}\quad s\le -1,\eqlab{psi}\\
 \psi(s)&=1,\quad \text{for all}\quad s\ge 1.\nonumber
\end{align}
Notice that these functions are not asymptotic to $0$ and $1$ but rather reach these values at finite values of $s$. Simple analytic functions like $\frac12 +\frac{1}{\pi}\arctan(s)$ or functions like \eqref{gkfunc}, see also \eqref{gkfunc0}, that appear in applications and satisfy (A1) do therefore not belong to this class. In examples, see e.g. \cite{bonet-rev2016a, krihog}, functions of the type \eqref{psi} are also always piecewise polynomial, the simplest example being (although this is clearly only $C^0$)
\begin{align*}
 \psi(s) = \begin{cases}
            1\,&\text{for $s\ge 1$},\\
            \frac12 s + \frac12\,&\text{for $s\in (-1,1)$},\\
            0\,&\text{for $s\le -1$}.
           \end{cases}
\end{align*}
Since the Sotomayor-Teixera regularization functions do not satisfy (A1) and do not appear naturally in applications, we shall not study these functions further in the present paper. 
}

\subsection{PWS systems}
Consider the PWS system \eqref{ztfpws}. 
Along the discontinuity set $\Sigma=\{(x,y)\vert y=0\}$, also called the switching manifold in the PWS literature \cite{Bernardo08}, $Z_\pm$ can either (a) be pointing in the same directions, (b) be pointing in opposite directions, or at least one of $Z_\pm$ is tangent. The subset $\Sigma_{cr}$ along which (a) occurs is called crossing, which is relatively ``harmless''. Here orbits of \eqref{ztf} follow the orbits of \eqref{ztfpws} obtained by gluing orbits together on either side. The subset $\Sigma_{sl}$ along which (b) occurs, on the other hand, is called sliding. Here solutions of \eqref{ztfpws} cannot be extended beyond the intersection with $\Sigma$. In the PWS literature, \eqref{ztfpws} is therefore frequently ``closed'' by subscribing a Filippov vector-field along $\Sigma$. See \figref{Filippov} for a geometric construction. Interestingly, under assumptions (A0) and (A1), see \cite{Buzzi06,Llibre07,llibre2009a,bonet-rev2016a,kristiansen2017a,kristiansen2018a}, the Filippov vector-field also coincides with a reduced vector-field on a critical manifold of \eqref{ztf} for $\epsilon=0$, obtained upon blowup of $\Sigma$.  

\begin{figure}
\begin{center}
\includegraphics[width=.7\textwidth]{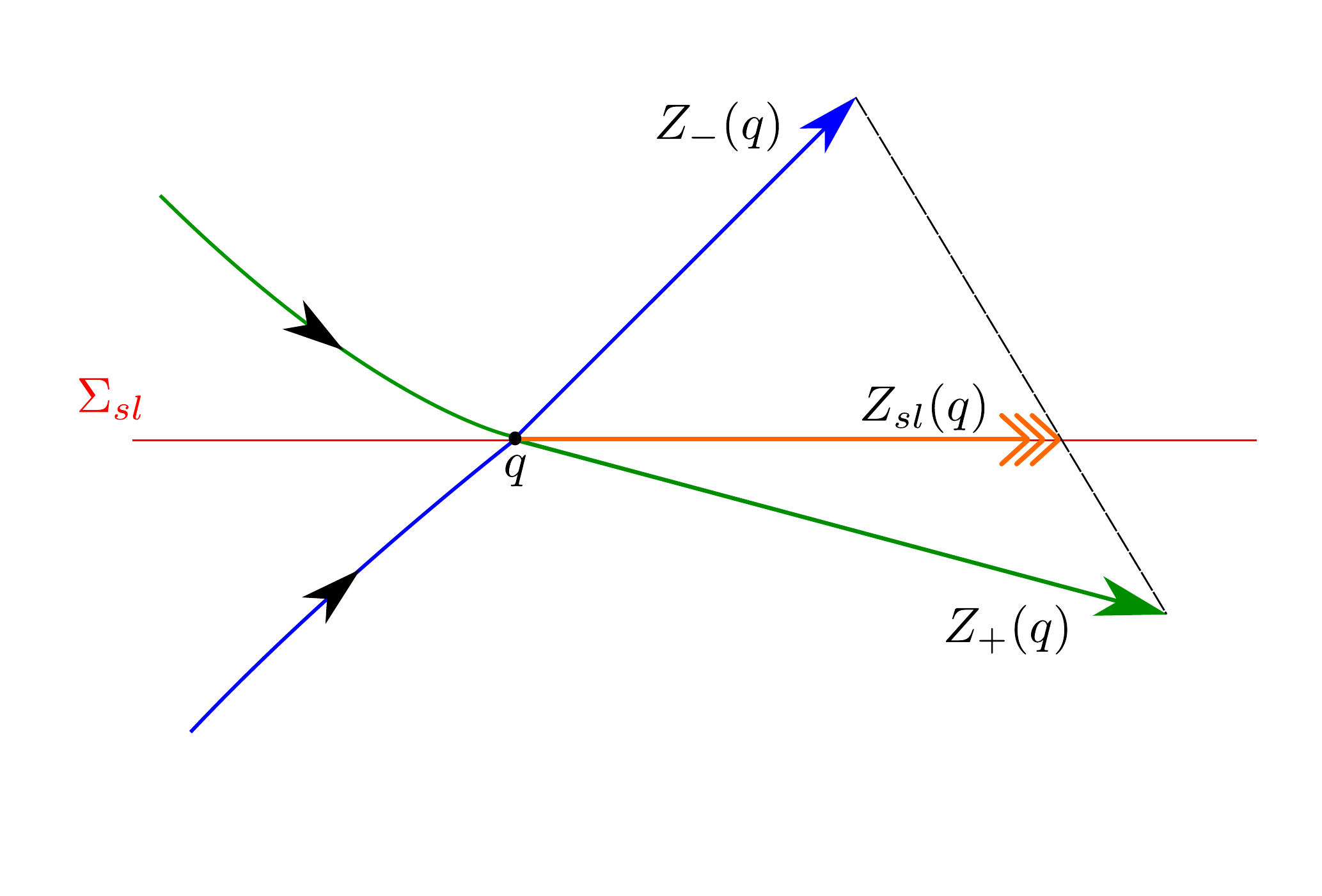}
 \caption{Geometric construction of the Filippov sliding vector-field $Z_{sl}$ as the convex combination of $Z_\pm$ such that $Z_{sl}$ is tangent to $\Sigma$.}
 \end{center}
 \figlab{Filippov}
              \end{figure}

It is possible to characterize crossing and sliding using the Lie derivative $Z_+ h(\cdot):=\nabla h(\cdot) \cdot Z_+(\cdot,\alpha)$ of $h(x,y)=y$, such that $\Sigma = \{h(x,y)=0\}$, along $Z_\pm$:
\begin{align*}
 \Sigma_{cr} &= \{q\in \Sigma\vert (Z_+h(q))(Z_-h(q))>0\},\\
 \Sigma_{sl} &= \{q\in \Sigma\vert (Z_+h(q))(Z_-h(q))<0\}.
\end{align*}
The tangencies
\begin{align*}
 T = \{q\in \Sigma\vert (Z_+h(q))(Z_-h(q)) = 0\},
\end{align*}
are in between.

\subsection{The visible fold tangency}
In \cite{bonet-rev2016a}, the authors also considered systems of the form \eqref{ztf} satisfying (A0). In particular, they considered the local behaviour near a visible fold tangency $T$, assuming that an orbit $\gamma$ of $Z_+$ had a quadratic tangency with $\Sigma$ at a point $q\in \Sigma$, while $Z_-(q)$ was transverse to $\Sigma$. See \figref{visible} for an illustration of the setting. Notice, the tangency is called visible because the orbit $\gamma$ is contained within $y\ge 0$. Using Lie derivatives, such a visible fold point can be written as
\begin{align*}
Z_+h(q) = 0,\,Z_+(Z_+ h)(q) > 0,Z_-h(q)>0.
\end{align*}
We consider the case illustrated in \figref{visible} where $\Sigma_{sl}$ and $\Sigma_{cr}$ occur at $x<0$ and $x>0$, respectively. 
Based on appropriate scalings, nonlinear transformations of time and the flow-box theorem, the authors of \cite{bonet-rev2016a} constructed a change of coordinates such that near $q$, the system could be brought into the form \eqref{fA} with
\begin{align}
Z_+(z) = \begin{pmatrix}
                 1+f(z)\\
                 2x + yg(z)
                \end{pmatrix},\quad 
                Z_-(z) =  \begin{pmatrix}
                           0\\
                           1
                          \end{pmatrix},\eqlab{ZpmNF}
\end{align}
where $f$ and $g$ are smooth and where $f(0)=0$ and $q=(0,0)$ in the new coordinates, suppressing any dependency on a parameter $\alpha$ in these expressions. The result is local, so we assume $z\in U_\xi:= [-\xi,\xi]^2$ with $\xi>0$ sufficiently small. See \cite[Proposition 14]{bonet-rev2016a}. Setting $f=g=0$ in \eqref{ZpmNF}, we realize that the orbit $\gamma$, which is tangent to $\Sigma$ at $(x,y)=(0,0)$, is close to the parabola $y=x^2$. In any case, it is locally a graph $y=\gamma(x)$, $x\in I\subset [-\xi,\xi]$, abusing notation slightly. It acts as a separatrix: Everything within $\{(x,y)\in U_\xi \vert x<0,0<y<\gamma(x)\}$ reaches $y=0$ and ``slides'', whereas everything above $y=\gamma(x)$ does not. See \figref{visible}. In fact, on $\Sigma_{sl}$, a simple calculation shows that the Filippov vector-field gives
\begin{align*}
 \dot x = \frac{1+f(x,0)}{1-2x},
\end{align*}
which is locally $\dot x\ge c>0$ for $c$ sufficiently small. This produces the local picture in \figref{visible}. 

The authors of \cite{bonet-rev2016a} analyse \eqref{ztf} with $Z_\pm$ as in \eqref{ZpmNF} using asymptotic methods, but considered, following \cite{Sotomayor96}, the special class of non-analytic regularization functions $\psi(s)$, recall \eqref{psi}.
The authors described the perturbation of a critical manifold and its extension by the forward flow into $y>0$ as $\epsilon\rightarrow 0$ for this class of functions. They also studied the case where a repelling limit cycle of $Z_+$ grazes $\Sigma$ and argued that this PWS bifurcation had to give rise to a saddle-node bifurcation of limit cycles for $\epsilon\ll 1$. But they did not proof this latter statement nor did they address the question of whether additional saddle-nodes could exist. 

\subsection{Main results}\seclab{mainResults}
In this paper, we will, following \cite[Proposition 14]{bonet-rev2016a} and the equations \eqref{ZpmNF}, revisit the results of \cite{bonet-rev2016a} within our slightly more general framework. 
\qs{First, we provide a detailed description (not available in \cite{bonet-rev2016a}) of a transition mapping near the visible fold, see \thmref{main1}. Next, we use this accurate description to provide a rigorous proof of existence of a unique saddle-node bifurcation in the situation, where the PWS system has a repelling limit cycles undergoing a ``grazing bifurcation'', see \thmref{main2}.}

In this paper, ``smooth'' will mean $C^l$ with $l$ sufficiently large. We will leave it to the reader to determine what ``sufficiently'' is for the various statements to come.


\subsection*{\qss{The blowup approach for \eqref{ztf}}}
\qs{Our approach to study systems of the form \eqref{ztf} follows \cite{kristiansen2018a} and is very general. 
 We consider the extended system
\begin{align}
 z' &=\epsilon Z(z,\phi(y\epsilon^{-1},\epsilon),\alpha),\eqlab{zext}\\
 \epsilon'&=0,\nonumber
\end{align}
with $z=(x,y)$,
obtained from \eqref{ztf} in terms of the ``fast time'': $()'=\epsilon \dot{(\,)}$. For this system, 
the set defined by $(x,y,0)$ is a plane of equilibria, with the subset given by $y=0$ being extra singular due to the lack of smoothness there. However, by the following final assumption (A2) on the regularization functions $\phi$, we gain smoothness by applying the blow-up transformation
\begin{align}
\eqlab{cyl_blowup_map1}
r\ge 0,\,(\bar y,\bar \epsilon)\in S^1 \mapsto \begin{cases}
y&=r\bar y,\\
\epsilon &=r\bar \epsilon,
\end{cases}
\end{align}
for all $x$, 
to the extended system \eqref{zext}. }
\begin{itemize}
\item[(A2)] There exists ``decay rates'' $k_\pm \in \mathbb N$, constants $c_j>0,j=0,\ldots,3$ and smooth functions $\phi_\pm:[0,c_0]\times [0,c_1]\rightarrow [c_2,c_3]$ such that 
\begin{align}
 \phi(\epsilon_1^{-1},r_1\epsilon_1) &= 1-\epsilon_1^{k_+}\phi_+(\epsilon_1,r_1),\eqlab{phiPA}\\
 \phi(-\epsilon_3^{-1},r_3\epsilon_3)&= \epsilon_3^{k_-}\phi_-(\epsilon_3,r_3),\eqlab{phiNA}
\end{align}
for all $(\epsilon_i,r_i)\in [0,c_0]\times [0,c_1]$ and $i=1,3$. 
\end{itemize}
Since $\epsilon\ge 0$, only $\bar \epsilon\ge 0$ is relevant. Geometrically, \eqref{cyl_blowup_map1} therefore blows up the line defined by $(x,0,0)$ to a (semi-)cylinder, see also \figref{visible1} below. 

The variables $(\epsilon_i,r_i)$, $i=1,3$ in \eqref{phiPA} and \eqref{phiNA} appear naturally from the blowup approach. Indeed, following e.g. \cite{krupa_extending_2001}, we describe the blowup transformation in directional charts, obtained by setting $\bar y=1$, $\bar \epsilon=1$ and $\bar y=-1$, respectively. This produces the following local forms
\begin{align}
 r_1\ge 0,\epsilon_1\ge 0\, &\mapsto \begin{cases} y&=r_1,\\
                                      \epsilon &=r_1\epsilon_1,
                                     \end{cases}\eqlab{r1epsilon1}\\
 r_2\ge 0,y_2\in \mathbb R\, &\mapsto  \begin{cases} y&=r_2y_2,\\
                                      \epsilon &=r_2,
                                     \end{cases}\eqlab{r2y2}\\
 r_3\ge 0,\epsilon_3\in \mathbb R\, &\mapsto  \begin{cases} y&=-r_3,\\
                                      \epsilon &=r_3\epsilon_3,
                                     \end{cases}\eqlab{r3epsilon3}
\end{align}
respectively, of \eqref{cyl_blowup_map1}. We will enumerate these three charts as $(\bar y=1)_1$, $(\bar \epsilon=1)_2$ and $(\bar y=-1)_3$, respectively, giving reference to how the charts are obtained and the subscripts used. Hence $(\epsilon_1,r_1)$ and $(\epsilon_3,r_3)$ in \eqref{phiPA} and \eqref{phiNA} are local coordinates in $(\bar y=1)_1$ and $(\bar y=-1)_3$, respectively. The coordinate changes between the different charts are obtained from the following equations:
\begin{align*}
\begin{cases}
 r_2 &= r_1\epsilon_1,\quad y_2 = \epsilon_1^{-1},\quad \,\,\,\,\mbox{for $\epsilon_1>0$},\\
 r_2 &= r_3\epsilon_3,\quad y_2 = -\epsilon_3^{-1},\quad\mbox{for $\epsilon_3>0$}.
 \end{cases}
\end{align*}

\qs{In contrast to the usual blowup approach \cite{dumortier_1996,krupa_extending_2001}, we will not divide by $r$ to obtain improved properties. Instead, we gain hyperbolicity and recover the piecewise smooth systems by dividing by common factors $\epsilon_i$, $i=1,3$, in the charts $(\bar y=1)_1$ and $(\bar y=-1)_3$, respectively. This is carefully described in \cite{kristiansen2018a}, see also \appref{appA}; notice e.g. \remref{yZp}.}

%
 

%
                        

%


\qs{Under assumption (A2), we also have that the perturbation for $y\ne 0$ is regular:
\begin{lemma}\lemmalab{regular}
Consider (A0), (A1) and (A2). 
 Fix any (small) constant $c>0$ and consider \eqref{ztf} within the set defined by $\vert y\vert \ge c>0$. This system is a regular perturbation of the PWS system \eqref{ztfpws} for $\epsilon=0$.  
 \end{lemma}
 \begin{proof}
  Consider $y\ge c$. The case $y\le -c$ is identical and therefore left out. Then by (A0), (A2) and \eqref{r1epsilon1} we can write \eqref{ztf}
  as
  \begin{align*}
\dot z &=\left(1-(y^{-1} \epsilon)^k \phi_+(\epsilon y^{-1},y)\right)Z_+(z,\alpha)+(y^{-1}\epsilon)^k \phi_+(\epsilon y^{-1},y)Z_-(z,\alpha),
  \end{align*}
  which within $y\ge c$ is a smooth perturbation of $\dot z=Z_+(z,\alpha)$, the perturbation being of $\mathcal O(\epsilon^k)$, uniformly for $y\ge c$,  as $\epsilon\rightarrow 0$.
 \end{proof}
}

\begin{remark}
Following \eqref{gklimit}, we may realise that \eqref{gkfunc} satisfies (A2) with $k_\pm = 1$ and 
\begin{align*}
 \phi_+ (0,0) = \phi_-(0,0) = 1.
\end{align*}
In later computations, we will use the following regularization function
\begin{align}\eqlab{phiexample}
 \phi(s,\epsilon) = \frac12 \left(1+\frac{s}{\sqrt{s^2+1}}\right)
\end{align}
for which
\begin{align*}
 \phi_+(\epsilon_1,r_1\epsilon_1 ) &=  \frac{\sqrt{1+\epsilon_1^2}-1}{2\sqrt{1+\epsilon_1^2}}=\frac{1}{4}\epsilon_1^2 +\mathcal O(\epsilon_1^4),
\end{align*}
and $\phi_-(\epsilon_3,r_3\epsilon_3 ) = \phi_+(\epsilon_3,r_3\epsilon_3)$. Hence $k_\pm = 2$ for \eqref{phiexample}. Functions like \eqref{phicor} and $\tanh$, where $k_\pm=\infty$ and hence excluded by (A2), are more difficult, because the blowup method has to be adapted to deal with the non-algebraic terms, see \cite{kristiansen2017a}. This is also the subject of the recent paper \cite{kristiansen2019d}.  
\end{remark}

\subsection*{A local transition map} 
In the following, we state our first main result.
Let $\delta\in (0,\xi)$ and consider two sections $\Sigma_{L}$ and $\Sigma_R$  transverse to $Z_+$ within $y=\delta\in (0,\xi)$ such that points in $\Sigma_L$ flow to points in $\Sigma_R$ in finite time by following the flow of $Z_+$. 
Specifically, we take
\begin{align}
 \Sigma_L &= \{(x,y)\vert y=\delta,\,x\in I_L\subset (-\xi,0)\},\eqlab{SigmaL0}\\
 \Sigma_R &= \{(x,y)\vert y=\delta,\,x\in I_R\subset (0,\xi)\},\eqlab{SigmaR0}
\end{align}
where $I_L$ and $I_R$ are closed intervals. By adjusting $\delta$, $\xi$, $I_L$ and $I_R$, if necessary, we may assume that $\gamma$ intersects $\Sigma_L$ and $\Sigma_R$ in their interior and that the $x$-values of the intersection, $\gamma_L$ and $\gamma_R$, respectively, satisfy $\gamma_L<0<\gamma_R$. See \figref{visible}. 
Then we define $Q(\cdot,\epsilon)$ as the transition map (of Dulac type) $I_L\ni x\mapsto Q(x,\epsilon)\in I_R$ obtained by the first intersection $(Q(x,\epsilon),\delta)\in \Sigma_R$ of the forward flow, defined by \eqref{fA}, with $Z_\pm$ as in \eqref{ZpmNF}, of the point $(x,\delta)\in \Sigma_L$. Since $k_-$ plays little role, recall (A2), we set $$k=k_+,$$ for simplicity in the following.


\begin{figure}
\begin{center}
\includegraphics[width=.7\textwidth]{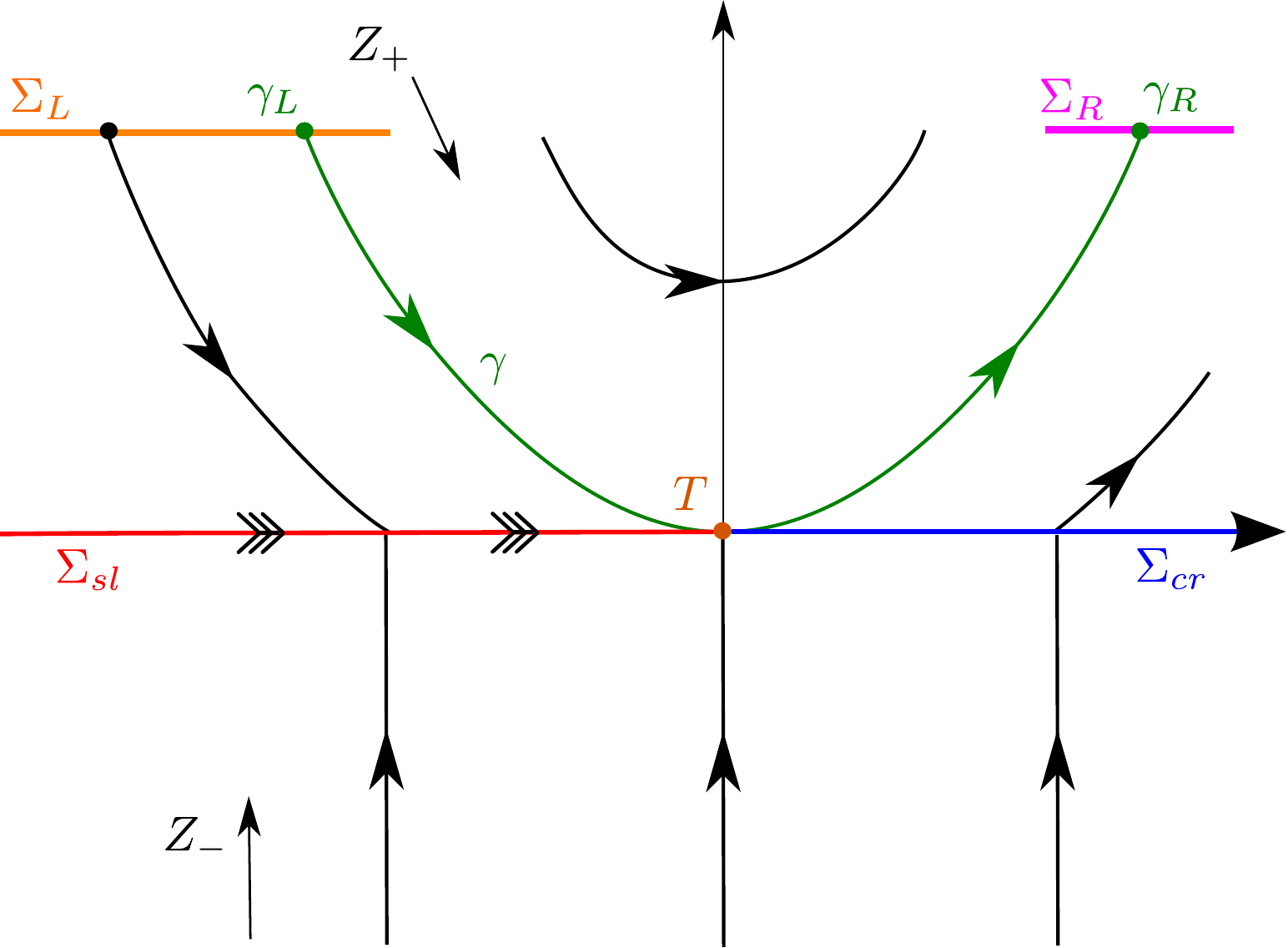}
  \end{center}\caption{The visible fold. $Z_+$ has a quadratic tangency with $\Sigma$ at $(x,y)=(0,0)$ while $Z_-(0,0)$ is transverse. Along the sliding region, $\Sigma_{sl}=\{x<0\}$, the Filippov vector-field gives $\dot x>0$. Every point in $\Sigma_L$ with $x<\gamma_L$ therefore reaches $\Sigma_R$ at $x=\gamma_R$ by following the Filippov flow. See also \remref{pwsmap}.  }
 \figlab{visible}
              \end{figure}

\begin{theorem}\thmlab{main1}
 Consider \eqref{ztf}, satisfying (A0), specifically \eqref{fA} with \eqref{ZpmNF}, and suppose (A1)-(A2). 
  \begin{itemize}
 \item[(a)] \qs{\cite{Buzzi06,Llibre07,llibre2009a,kristiansen2018a}} Fix any $0<\nu<\xi$ and let $J=[-\xi,-\nu]$. Then there exists an $\epsilon_0>0$ such that for any $\epsilon\in (0,\epsilon_0)$, there exists a locally invariant manifold $S_{\epsilon}$ as a graph over $J$:
 \begin{align*}
  S_{\epsilon}:\,y = \epsilon h(x,\epsilon),\,x\in J,
 \end{align*}
where $h$ is smooth in both variables. The manifold has an invariant Lipschitz foliation of stable fibers along which orbits contract exponentially towards $S_{\epsilon}$. For $\epsilon=0$ these fibers coincide with the orbits of $Z_\pm$ reaching $\Sigma\cap \{x\in J\}$ after a finite time. Moreover, $Z\vert_{S_\epsilon}$ is a regular perturbation of the Filippov vector-field.
\item[(b)] The forward flow of $S_{\epsilon}$ intersects $\Sigma_R$ in $(m(\epsilon),\delta)$ where 
\begin{align}
 m(\epsilon) = \gamma_R+\epsilon^{2k/(2k+1)} m_1(\epsilon),\eqlab{mEpsSM}
\end{align}
with $m_1$ continuous. 
\item[(c)] Fix $\theta>0$ so small that $0<\theta-\gamma_L<\xi$ and let $K=[-\xi,\gamma_L-\theta]\subset I_L$. Consider $Q_K(\cdot,\epsilon) = Q(\cdot,\epsilon)\vert_K:K\rightarrow I_R$. Then $Q_K$ is a strong contraction:
\begin{align*}
 Q_K(x,\epsilon ) = m(\epsilon)+\mathcal O(e^{-c/\epsilon}),\quad \partial_x Q_K(x,\epsilon) = \mathcal O(e^{-c/\epsilon}).
\end{align*}
\item[(d)] Fix $\upsilon \in (-1,0)$. Then for any $c>0$ sufficiently small, there exist positive numbers $\epsilon_0$, $\delta,\xi$, $\chi$ and intervals $I_L$ and $I_R$ such that the following holds for all $0<\epsilon\le \epsilon_0$.
\vspace{0.3cm}
\begin{itemize}
 \setlength\itemsep{1em}
\item[(i)] $\vert Q'_x(x,\epsilon)\vert \le c$ for all $x\in I_L\cap \{x\le \gamma_L-\chi\epsilon^{2k/(2k+1)}\}$.
\item[(ii)] $Q'_x(x,\epsilon)<0,\,Q''_{xx}(x,\epsilon)<0$ for all $x\in I_L^C:=I_L\cap \{\vert x-\gamma_L\vert \le \chi \epsilon^{2k/(2k+1)}\}$. \qs{In particular, there exists a unique $x\in I_L^C$ such $Q'_x(x,\epsilon)=\upsilon$ for which
\begin{align}
Q''_{xx}(x,\epsilon)\le -c^{-1}.\eqlab{concaveDown}
\end{align}
}
 \item[(iii)] $1-c\le \vert  Q'_x(x,\epsilon)\vert \le 1+c$ for all $x\in I_L\cap \{x\ge \gamma_L+\chi \epsilon^{2k/(2k+1)}\}$.
\end{itemize}
\vspace{0.3cm}
\end{itemize}

\end{theorem}
\begin{remark}\remlab{ref1}
 \qs{As highlighted, \thmref{main1} (a) is known from many references, see e.g. \cite{Buzzi06,Llibre07,llibre2009a}, in related and even more general contexts. Notice that the existence of the invariant manifold follows directly by inserting the scaling $y=\epsilon y_2$. In terms of $(x,y_2)$ the system is slow-fast such that Fenichel's theory \cite{fen1,fen2,fen3} applies, see also \eqref{scalingEqns} in \appref{appA}. This produces $S_\epsilon$ as a slow manifold, having a smooth foliation of stable fibers within compact subsets of the $(x,y_2)$-system. For the Sotomayor-Teixera regularization functions, where one can restrict the $(x,y_2)$-system to a compact strip $y_2\in[-1,1]$, recall \eqref{psi}, this produces \thmref{main1}(a). For the ``asymptotic'' regularization functions used in the present paper, the properties of the stable foliation is slightly more technical. For this reason, and for the sake of readability and completeness, we have decided to include (a) in the main theorem.
 }
 
 \qs{More specifically, we emphasize that the results in \cite[Theorem 2.1]{bonet-rev2016a}, valid for the Sotomayor-Teixera regularization functions \eqref{psi}, are similar to (a), (b) and (c). Notice, however, that the remainder of $m(\epsilon)$ in the setting of \cite{bonet-rev2016a} is $\mathcal O(\epsilon)$ (whereas it is $\mathcal O(\epsilon^{2k/(2k+1)})$ in \eqref{mEpsSM}). Nevertheless, the details of the mapping $Q$ in \thmref{main1} (d), covering a full neighborhood of $\gamma_L$, is not available in \cite{bonet-rev2016a}. It is this detailed information that we will use in the following to prove rigorous statements about the grazing bifurcation. }

\end{remark}
%

\begin{remark}\remlab{pwsmap}
 Notice, it is possible to obtain a ``singular'' map $Q_0:\Sigma_L\rightarrow \Sigma_R$ of the PWS Filippov system. This mapping is only continuous being of the following form
\vspace{0.3cm}
\begin{itemize}
 \setlength\itemsep{1em}
\item[(i)] $Q'_{0}(x)=0$ for all $x<\gamma_L$;
\item[(ii)] $Q'_{0}$ not defined for $x=\gamma_L$;
 \item[(iii)] $1-c\le \vert  Q'_{0}(x)\vert \le 1+c$ for all $x>\gamma_L$.
\end{itemize}
\vspace{0.3cm} 
This holds for any $c>0$ provided that $\delta$ and $I_L$ are appropriately adjusted.
Compare with \thmref{main1}(d). (i) is due to the fact that every point in $I_L$ with $x<\gamma_L$ reaches the sliding segment, see \figref{visible}. Hence:
\begin{align*}
Q_0(x) = \gamma_R,
\end{align*}
for all $x<\gamma_L$. 


\end{remark}

\subsection*{Application to a grazing bifurcation} 
To present our second main result, we now assume the following:
\begin{itemize}
 \item[(B1)] Suppose that $Z_+$ has a hyperbolic and repelling limit cycle $\Gamma_0$ for $\alpha=0$ with a unique quadratic tangency with $\Sigma=\{y=0\}$ at the point $q=(0,0)$.
 \end{itemize}
Since $\Gamma_0$ is hyperbolic there exists a local family $\{\Gamma_\alpha \}_{\alpha\in I}$, where
\begin{align}
I=(-a,a), \quad a>0,\eqlab{IInt}  
\end{align}
of hyperbolic and repelling limit cycles of $Z_+$. 
\begin{itemize}
 \item[(B2)] Let $Y(\alpha)=\text{min}\, y(t)$ along $\Gamma_\alpha$ so that $Y(0)=0$ by assumption (B1). Suppose that 
 \begin{align*}
  Y'(0) >0.
 \end{align*}
\end{itemize}
By (B1) and (B2) and the implicit function theorem, $\Gamma_\alpha$ therefore, for $a>0$ sufficiently small, recall \eqref{IInt}, intersects $\{y=0\}$ only for $\alpha\le 0$, doing so twice for $\alpha<0$ and once for $\alpha=0$.
Finally:
\begin{itemize}
 \item[(B3)] Suppose that $Z_-$ has a positive $y$-component at $(x,y)=(0,0)$ for $\alpha=0$, i.e. $Z_-f(0,0)>0$. 
\end{itemize}
We illustrate the setting in \figref{grazing}. 
As a consequence of (B1) and (B3), and the implicit function theorem, the PWS system $(Z_-,Z_+)$ has a visible fold near $(x,y)=(0,0)$ for all $\alpha \in I$ (after possibly restricting $a>0$ further). In fact, also by the implicit function theorem, the $x$-value of this fold point depends smoothly on $\alpha$ and we can therefore shift it to $(x,y)=(0,0)$ for all $\alpha$. Moreover by (A3), we can apply the result of \cite{bonet-rev2016a} and bring the PWS system into the form \eqref{ZpmNF} near the fold. We will now study the bifurcation of limit cycles that occur for \eqref{fA} near $\alpha=0$ for all $0<\epsilon\ll 1$. (In the PWS setting, this bifurcation is known as the grazing bifurcation, see e.g. \cite[Fig. 14, section 4.11]{Kuznetsov2003}.) \qs{For this we study the Poincar\'e map (of Dulac type) $P(\cdot,\alpha,\epsilon):I_L\rightarrow I_L$ obtained by the forward flow. This mapping is well-defined by the assumptions (B1)-(B3) and by \thmref{main1}, based on assumptions (A1)-(A2). We compose $P(\cdot,\alpha,\epsilon)$ into two parts: The ``local'' mapping $Q(\cdot,\alpha,\epsilon):I_L\rightarrow I_R$, studied in \thmref{main1}, and a ``global'' mapping $R(\cdot,\alpha,\epsilon):I_R\rightarrow I_L$:
\begin{align}
 P(x,\alpha,\epsilon) = R(Q(x,\alpha,\epsilon),\epsilon,\alpha).\eqlab{QR}
\end{align}
By (A2), recall \lemmaref{regular}, $x\mapsto R(x,\alpha,\epsilon)$ is a regular perturbation of the associated mapping $x\mapsto R(x,0,\alpha)$ obtain from the $Z_+$ system.}
\begin{lemma}\lemmalab{Rmap}
 Assume (B1) and (B2). The mapping $R$ is smooth in all of its arguments. Also there exists a $\omega>0$ such that upon decreasing $\xi$ and $\delta$, if necessary, the map satisfies:
 \begin{align}
  R(\gamma_R,0,0) &= \gamma_L,\eqlab{Rprop1}\\
  R'_x(\gamma_R,0,0)&<-1-\omega,\eqlab{Rprop2}\\
  R'_\alpha(\gamma_R,0,0)&>\omega. \eqlab{Rprop3}
 \end{align}
\end{lemma}
\begin{proof}
\eqref{Rprop1} holds by assumption (B1) and the definition of $\gamma_L$ and $\gamma_R$. By (B1), $\Gamma_0$ is a hyperbolic but repelling limit cycle of $Z_+$. Therefore $P'_x(\gamma_R,0,0)>1$, as a mapping obtained from $Z_+$ at $\epsilon=0$ only, and hence by decomposing $P$ into $R$ and a local map $\widetilde Q:\Sigma_L\rightarrow \Sigma_R$, say, we obtain, upon restricting $\xi$ and $\delta$, that $-\widetilde Q$ is as close to the identity as desired. Indeed, as a mapping obtained from the flow of $Z_+$, $\widetilde Q$ is regular and obtained by a short integration time. The integration time can be decreased by decreasing $\delta$. By the chain rule, we therefore obtain \eqref{Rprop2}. Finally, \eqref{Rprop3} follows from (B2) using \eqref{QR} and the fact that $-\widetilde Q$ is close to the identity map. We leave out the simple details. 

\end{proof}

\begin{figure}
\begin{center}
\includegraphics[width=.65\textwidth]{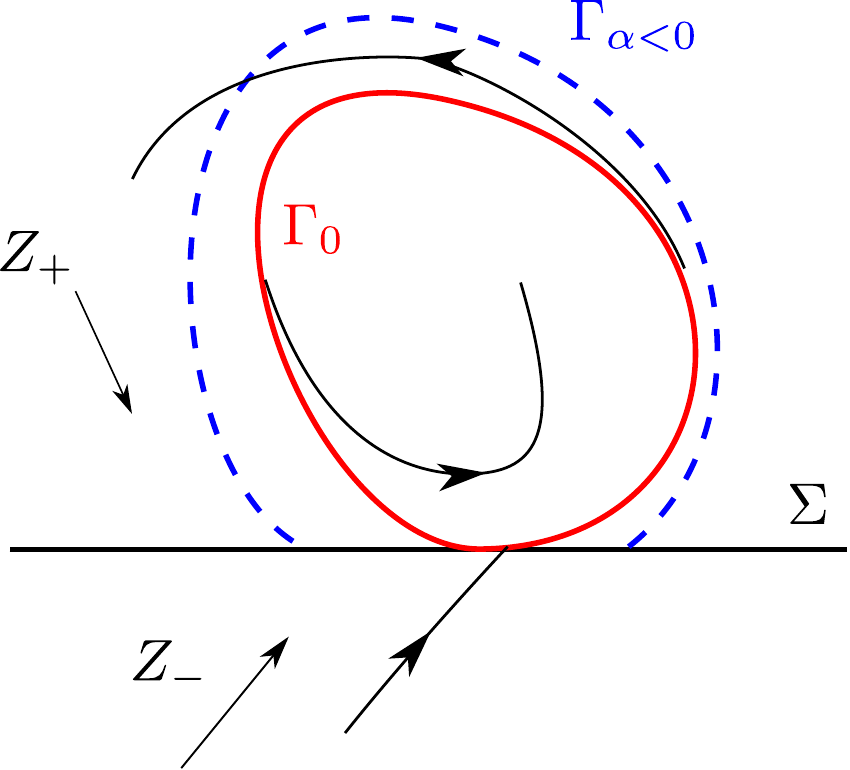}
 \end{center}
\caption{A grazing limit cycle for $\alpha=\epsilon=0$. Assumption (B1) is so that $\Gamma_0$ is repelling for $Z_+$ for $\alpha=0$. By (B2), $\Gamma_{\alpha<0}$ (in blue) locally intersects $\Sigma$ twice near the fold. Black orbits are (backwards) transients for $\alpha=0$, demonstrating the repelling nature of $\Gamma_0$.}
 \figlab{grazing}
              \end{figure}

              \qs{Before stating our theorem on the grazing bifurcation, we recall that a saddle-node bifurcation of a periodic orbit is a saddle-node bifurcation of a fix-point of the Poincar\'e map $P(\cdot,\alpha,\epsilon)$, see \cite[Theorem 1, p. 369]{perko2001a}. By \eqref{QR}, we can write the fix-point equation $P(x,\alpha,\epsilon)=x$ as follows:
              \begin{align}
               Q(x,\alpha,\epsilon) = R^{-1}(x,\alpha,\epsilon),\eqlab{QRFixPoint}
              \end{align}
where $R^{-1}(\cdot,\alpha,\epsilon):I_R\rightarrow I_L$ is the inverse of $R(\cdot,\alpha,\epsilon)$. The saddle-node bifurcation we will describe will be unfolded by the parameter $\alpha$. Using \eqref{QRFixPoint}, it is then straightforward to show that the conditions 
\begin{align}
 Q'_x &=(R^{-1})'_x,\quad (\text{Degeneracy condition})\eqlab{QRcond0}\\
 Q''_{xx} &\ne (R^{-1})''_{xx},\quad (\text{Nondegeneracy condition I})\eqlab{QRcond1}\\
 Q'_{\alpha}& \ne (R^{-1})'_{\alpha},\quad (\text{Nondegeneracy condition II})\eqlab{QRcond2}
\end{align}
are sufficient for a saddle-node bifurcation of $P$ at $(x,\alpha,\epsilon)$ satisfying \eqref{QRFixPoint}. 
Here all partial derivatives $()'_x =\frac{\partial}{\partial x},\,()''_{xx} =\frac{\partial^2}{\partial x^2},\,()'_{\alpha} =\frac{\partial}{\partial \alpha}$ in \eqref{QRcond0}, \eqref{QRcond1} and \eqref{QRcond2} are evaluated at the bifurcation point.  We will refer to the bifurcating nonhyperbolic periodic orbit as the ``saddle-node periodic orbit''.
}

      \qs{ We also recall the definition of the Hausdorff distance between two non-empty compact subsets $X$ and $Y$:
             \begin{align*}
              \text{dist}_{\text{Hausdorff}}(X,Y) = \text{max}\,\left\{\sup_{x\in X}\inf_{y\in Y}d(x,y),\, \sup_{y\in Y}\inf_{x\in X}d(x,y)\right\}.
             \end{align*}
             Here $d(x,y)$ is the distance between two points $x,y\in \mathbb R^2$. 
             $\text{dist}_{\text{Hausdorff}}$ turns the set of non-empty compact subsets into its own metric space \cite{munkres2000a}. }
              
We now have.
\begin{theorem}\thmlab{main2}
 Suppose (A0)-(A2) and (B1)-(B3). Then there exists a locally unique saddle-node bifurcation of limit cycles for all $0<\epsilon\ll 1$ at 
 \begin{align}
  \alpha = \epsilon^{2k/(2k+1)} \alpha_2(\epsilon),\eqlab{alphaEpsSN}
 \end{align}
with $\alpha_2$ continuous, such that limit cycles only exist within $\alpha\in I$ for $$\alpha\le \epsilon^{2k/(2k+1)} \alpha_2(\epsilon),$$ two for $\alpha<\epsilon^{2k/(2k+1)} \alpha_2(\epsilon)$ and precisely one for $$\alpha= \epsilon^{2k/(2k+1)} \alpha_2(\epsilon).$$ The saddle-node periodic orbit for $\alpha=\epsilon^{2k/(2k+1)} \alpha_2(\epsilon)$ converges in Hausdorff distance to the grazing limit cycle $\Gamma_0$ of $Z_+$ as $\epsilon\rightarrow 0$. 
\end{theorem}

\begin{remark}\remlab{discontmap}
\qss{
The grazing bifurcation for the discontinuity system has been studied by many authors, also in cases where the codimension of $\Sigma$ is greater than one, see e.g. \cite{Bernardo08} and references therein. In this context, the grazing bifurcation (in \cite{Bernardo08}: the grazing-sliding bifurcation) can be studied using formal normal forms of discontinuous return mappings, analogously to $Q_0$ in \remref{pwsmap}, see e.g. \cite[Theorem 8.3 and section 8.5.3]{Bernardo08}. \cite{Bernardo08} also describes -- using these normal forms and results on border collision bifurcations for PWS maps, see e.g. \cite[Theorem 3.4]{Bernardo08} -- how  the grazing bifurcation in higher dimensions can lead to emergence of a chaotic attractor. The saddle-node, described in \thmref{main2} for the smooth system \eqref{ztf}, is in the discontinuous case called a ``non-smooth fold'' in \cite{Bernardo08}. }
\end{remark}

\subsection{Overview}\seclab{overview}
\qss{
In \secref{model} we present an example where \thmref{main2} can be applied and provide some numerical comparisons.  We prove \thmref{main1} in \secref{proof1}. Given the blowup transformation defined by \eqref{cyl_blowup_map1}, the proof will rest upon another blowup of a nonhyperbolic point $\overline T$ -- being the imprint of the PWS visible fold -- which we describe in the chart $(\bar y=1)_1$. Then, upon undertaking a careful blowup analysis, see \secref{blowupT}, \secref{r1} and \secref{eps1}, we combine the findings in \secref{main1bc} and \secref{chini} to prove \thmref{main1} (b), (c) and (d), respectively. \thmref{main1} (a), being standard, is moved to the \appref{appA}.  \thmref{main2} is proven in \secref{proof2} using the implicit function theorem and the details of the local transition map described in \thmref{main1}(d). We conclude the paper in \secref{disc}. Here we discuss the assumptions, the regularization functions used, possible extensions to our work and compare our results with \cite{bonet-rev2016a}.  }
 \section{The friction oscillator}\seclab{model}
  Systems of the form \eqref{ztf} often appear in models of friction. Consider for example the system in \figref{model}(a), where a mass-spring system is on a moving belt. This produces the following equations
  \begin{align}
   \dot x &=y-\alpha,\eqlab{friction}\\
   \dot y &=-x - \mu(y,\phi(y\epsilon^{-1},\epsilon)),\nonumber
  \end{align}
where $\alpha>0$ is the belt speed, $x$ is the elongation of the spring and $y$ is the velocity relative to the belt, all in nondimensional form. Furthermore, $\mu$ is the friction force opposing the relative velocity, i.e. $\mu>0$ for $y>0$ and $\mu<0$ for $y<0$. Many different forms of $\mu$ exists, often PWS, but we will suppose that
%
\begin{align}
\mu(y,p) = \mu_+(y)p - \mu_+(-y)(1-p),\eqlab{muyp}
\end{align}
as  desired, such that $\mu$ is ``odd'' with respect to $(y,p)\mapsto (-y,1 -p)$. Here $\mu_+(y)$ is a smooth function having a minimum at $y=y_0>0$, see \figref{model}(b), such that 
\begin{align}
 \mu_+'(y_0)=0,\,\mu_+''(y_0)>0,\eqlab{y0}
\end{align}
and $\mu_+'(y)<0$ for all $y\in [0,y_0)$ while $\mu_+'(y)>0$ for all $y\in (y_0,\infty)$. The resulting shape of $\mu_+$ is shown in \figref{model}(b); the initial negative slope is known as the Stribeck effect of friction, see e.g. \cite{berger2002a}. 
In this way, we obtain the following associated PWS system
\begin{align}
 Z_+(x,y,\alpha) &= \begin{pmatrix}
                    y-\alpha\\
                    -x-\mu_+(y)
                   \end{pmatrix},\quad 
                   Z_-(x,y,\alpha) &= \begin{pmatrix}
                    y-\alpha\\
                    -x+\mu_+(-y)
                   \end{pmatrix}.\eqlab{ZNegHere}
\end{align} The system \eqref{muyp} with $p=\phi(y/\epsilon,\epsilon)$, $\phi$ satisfying (A1)-(A2), can viewed as a regularization of the PWS model $(Z_+,Z_-)$, given by \eqref{ZNegHere}, with the PWS friction law
\begin{align*}
 \mu(y) = \begin{cases}
  \mu_+(y) & \text{for $y>0$},\\
  -\mu_+(-y) & \text{for $y<0$}.
 \end{cases}
\end{align*}

Consider now $Z_+$. By \eqref{y0}, this system clearly has a Hopf bifurcation for $\alpha=y_0$ at $(x,y)=(-\mu_+(y_0),y_0)$. A straightforward calculation also shows that the Lyapunov coefficient is proportional to $\mu_+'''(y_0)$; the bifurcation being subcritical (supercritical) for $\mu_+'''(y_0)<0$ ($\mu_+'''(y_0)>0$, respectively). 
Suppose the former. Then for $y_0$ sufficiently small, it follows that the unstable Hopf limit cycles of $Z_+$ for $\epsilon=0$ intersect the switching manifold $y=0$ in the way described in (B1)-(B2) for some value of $\alpha=\alpha_*>y_0$ near $y_0$. The fact that $\alpha_*>y_0$ is due to the fact that the local Hopf limit cycles are repelling. Furthermore, the visible fold tangency with $y=0$ for $\alpha=\alpha_*$ occurs at the point $q:\,(x,y)=(-\mu_+(0),0)$. To verify (B3), notice by \eqref{ZNegHere} that $\dot y=2\mu_+(0)>0$ at $q$ from below. As a result, assuming (A1)-(A2), there exists saddle-node bifurcation of limit cycles near $\alpha=\alpha_*$ for $\epsilon\ll 1$, see \thmref{main2}. We collect the result in the following corollary.
\begin{cor}\corlab{model}
 Consider \eqref{friction} with $\mu$ of the form \eqref{muyp}, where there exists an $y_0>0$ such that \eqref{y0} holds and suppose that the regularization function $\phi$ satisfies (A1)-(A2). Suppose also that $\mu_+'''(y_0)<0$. Then for $y_0$ sufficiently small there exists an $\epsilon_0>0$ such that for every $\epsilon\in (0,\epsilon_0)$ the following holds:
 \begin{enumerate}
    \item There exists a subcritical Hopf bifurcation at $\alpha_H(\epsilon) = y_0+\mathcal O(\epsilon)$. 
  \item The unstable Hopf limit cycles undergo a locally unique saddle-node bifurcation at $\alpha_{SN}(\epsilon) = \alpha_*+\epsilon^{2k/(2k+1)}\alpha_2(\epsilon)>\alpha_H(\epsilon)$, $\alpha_*>y_0$ being the value for which a Hopf cycle of $Z_+$ grazes $\Sigma$. 
  \item For any $\alpha\in (\alpha_H(\epsilon),\alpha_{SN}(\epsilon))$ two (and, locally, only two) limit cycles exist: $\Gamma_{sl}(\alpha,\epsilon)$ and $\Gamma_{+}(\alpha,\epsilon)$, where:
  \begin{itemize}
  \item  $\Gamma_{sl}$ is hyperbolic and attracting.
  \item $\lim_{\epsilon\rightarrow 0}\Gamma_{sl}(\alpha,\epsilon)$ has a sliding segment.
  \item  $\Gamma_{+}$ is hyperbolic and repelling.
  \item $\lim_{\epsilon\rightarrow 0}\Gamma_{+}(\alpha,\epsilon)$ is a limit cycle of $Z_+$ contained within $y>0$. 
   \end{itemize}
   No limit cycles exist near $(x,y)=(-\mu_+(y_0),y_0)$ for $\alpha>\alpha_{SN}(\epsilon)$. 
   \item Let $\Gamma_{SN}(\epsilon)$ be the saddle-node periodic orbit for $\alpha=\alpha_{SN}(\epsilon)$. Then $\Gamma_{SN}(\epsilon)$ converges in Hausdorff-distance to the repelling limit cycle of $Z_+$ which grazes $y=0$ for $\alpha=\alpha_*$.
 \end{enumerate}
%
 \end{cor}
\begin{proof}
(1) and (2) follows from the analysis preceding the corollary. (3) and (4) are also consequences of the proof of \thmref{main2}, recall also \remref{pwsmap}. 
\end{proof}
\qs{
\begin{remark}\remlab{frictionrem}
\corref{model} applies to all friction models of the type shown in \figref{model}(b), satisfying $\mu_+'''(y_0)<0$, and to the general regularization functions satisfying (A1)-(A2). This shows that the saddle-node bifurcation in the friction oscillator problem is a very ``robust phenomena''. In fact, the details that depend upon the regularization function (like $k$) are ``microscopic'' (i.e. hidden in remainder terms $o(1)$). We discuss the friction oscillator problem further in the final paragraph of \secref{disc}.
\end{remark}}

It is known from experiments that subcritical Hopf bifurcations do occur for certain friction characteristics, see e.g. \cite{heslot1994a}. Explicitly, they occur for the model
\begin{align}
 \mu_+(y) = \mu_m + (\mu_s-\mu_m) e^{-\rho y}+cy,\eqlab{muModel}
\end{align}
proposed in \cite{berger2002a} and also studied in \cite{won2016a,papangelo2017a}, which we will now use in numerical computations. For \eqref{muModel}, $\mu_+(0)=\mu_s$ and $\mu_s>\mu_m>0$, $\rho>0$, and $c\in (0,\rho (\mu_s-\mu_m))$ for the Stribeck effect and the existence of $y_0$ to be present. In fact, for \eqref{muModel}, 
\begin{align*}
 y_0 &= -\rho^{-1} \log \left(c/(\rho (\mu_s-\mu_m))\right),\quad \mu_+''(y_0)=\rho c>0,\quad  \mu_+'''(y_0)=-\rho^2 c<0.
\end{align*}

In \figref{auto}, we illustrate numerical results, obtained using AUTO, for \eqref{muModel} with the following parameters:
\begin{align*}
 \mu_s= 1,\,\mu_m = 0.5,\,\rho=4,\,c=0.85,
\end{align*}
such that $y_0\approx 0.33$. In \figref{auto}, we have also used the regularization function \eqref{phiexample}
for which $k=k_\pm=2$ in (A2), and varied the small parameter $\epsilon$. In \figref{auto}(a), for example, a bifurcation diagram is shown using $\text{min}\,y$ as a measure of the amplitude, with $\epsilon$ varying along the different branches, highlighted in different colours. The Hopf bifurcation occurs at $\alpha\approx y_0$ with $\min\,y$ decreasing from around that same value (not visible in the zoomed version of the diagram in (a)). However, along each branch, a saddle-node bifurcation is visible. In black dashed lines is the unperturbed bifurcation diagram for $Z_+$. Numerically, we therefore see that the saddle-node bifurcation approaches the singular limit, in agreement with \thmref{main2}. See further details in the figure caption. In \figref{auto}(d), we show the value of $\alpha_*-\alpha$ along the saddle-node bifurcation for varying values of $\epsilon$ using a loglog-scale. Here $\alpha_*\approx 0.4$ is the unperturbed value of the bifurcation, where the limit cycle of $Z_+$ grazes the discontinuity set. The slope of the curve is almost constant; using least square we obtain a slope $\approx 0.8024$ which is also in agreement with \thmref{main2} for $k=2$; notice $\epsilon^{2k/(2k+1)} = \epsilon^{4/5}=\epsilon^{0.8}$ for this value of $k$. 

In \figref{auto}(c), the nonhyperbolic saddle-node periodic orbits are shown for different values of $\epsilon$. The dashed black curve (barely visible, but it has the largest amplitude) shows the grazing limit cycle for $Z_+$ at $\alpha=\alpha_*\approx 0.4$. Finally, \figref{auto}(d) shows two co-existing limit cycles for $\alpha=0.38$ and $\epsilon=5\times 10^{-4}$ in red. For comparison, the bifurcating limit cycle at this $\epsilon$-value and $\alpha=0.398$ is shown using a red dashed line. 

\begin{figure}
\begin{center}
\subfigure[]{\includegraphics[width=.49\textwidth]{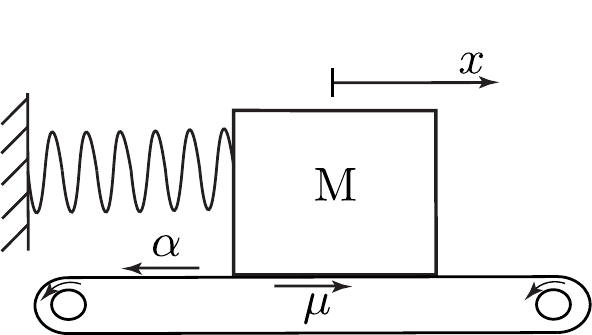}}
\subfigure[]{\includegraphics[width=.49\textwidth]{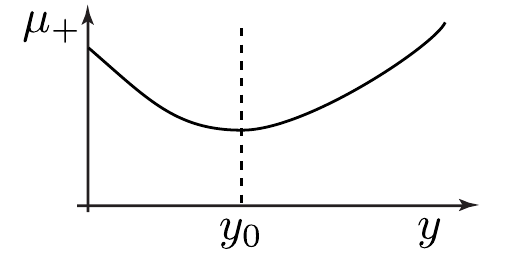}}
 \end{center}
 \caption{In (a): The mass-spring system on a moving belt. In (b): A Stribeck friction law with a minimum at $y=y_0$. }
  \figlab{model}
              \end{figure}
              
              \begin{figure}
\begin{center}
\subfigure[]{\includegraphics[width=.49\textwidth]{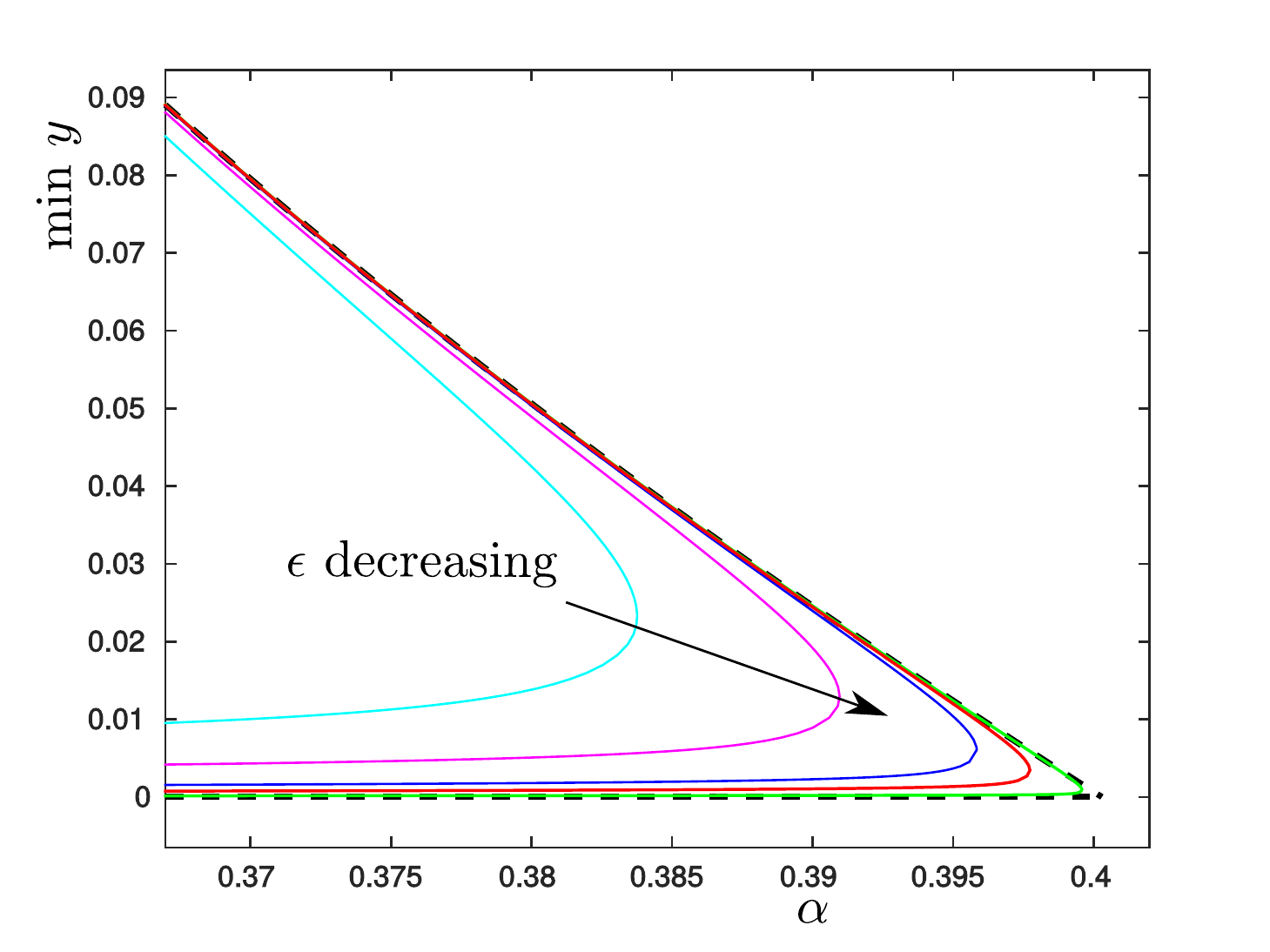}}
\subfigure[]{\includegraphics[width=.49\textwidth]{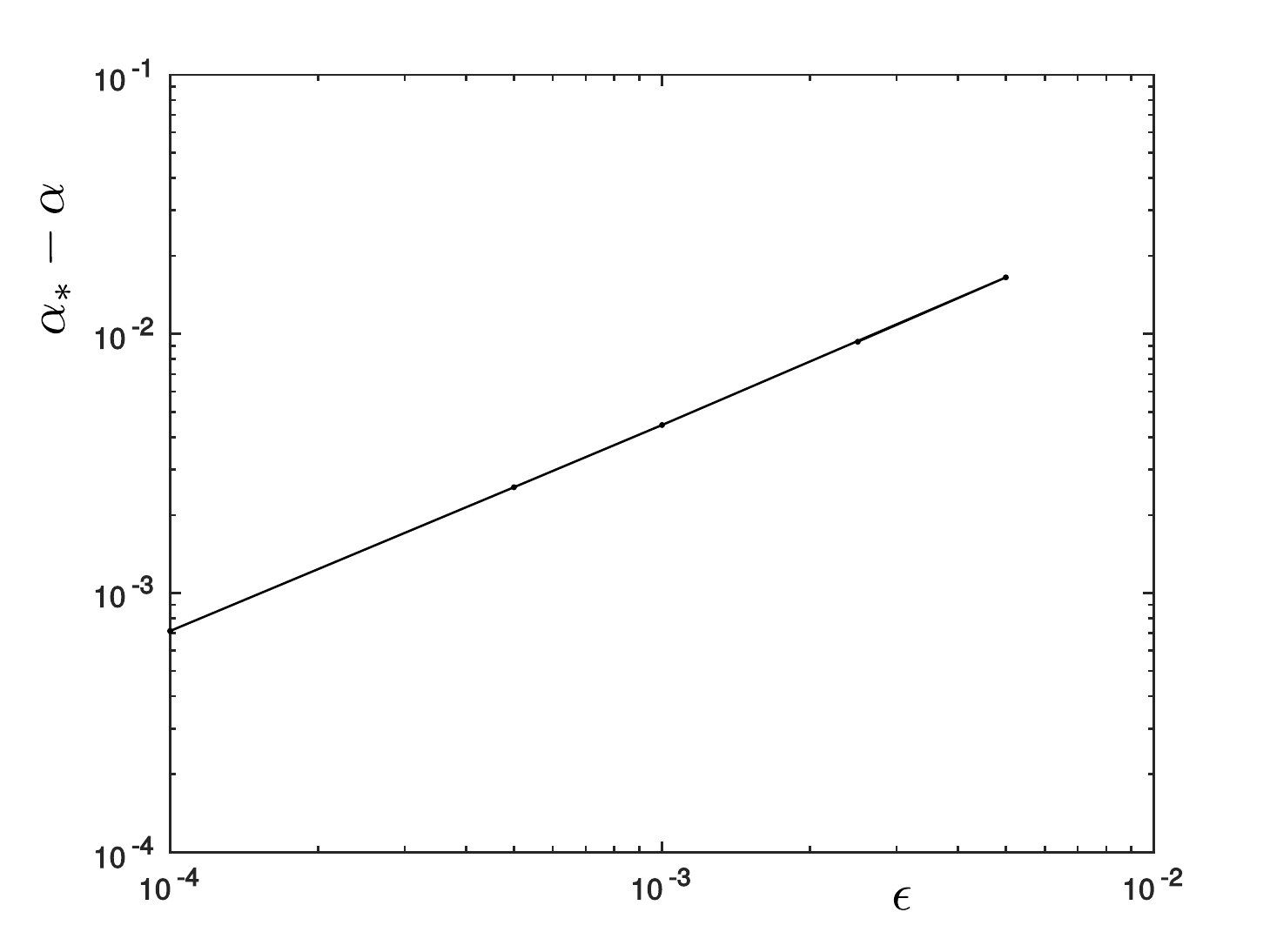}}
\subfigure[]{\includegraphics[width=.49\textwidth]{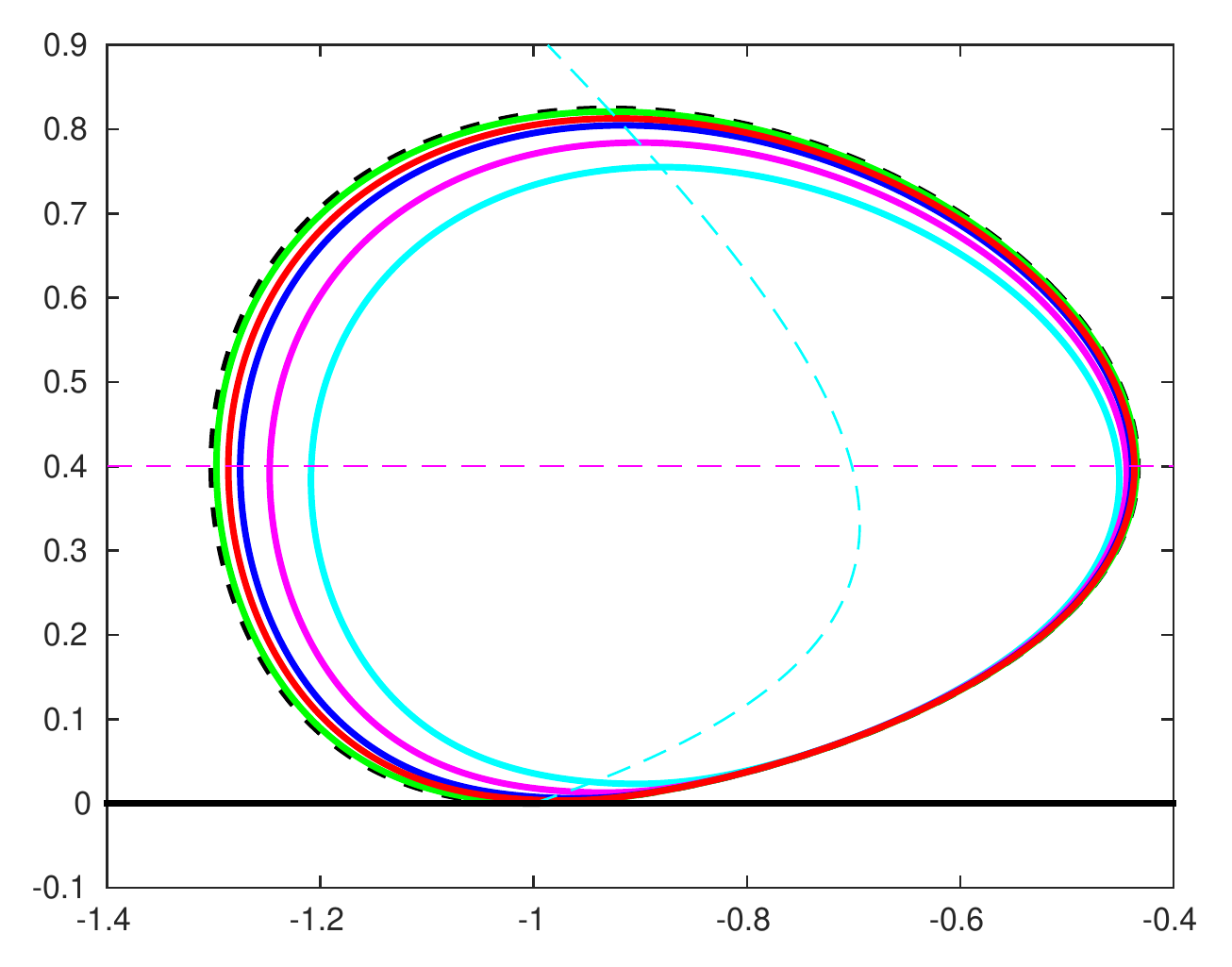}}
\subfigure[]{\includegraphics[width=.49\textwidth]{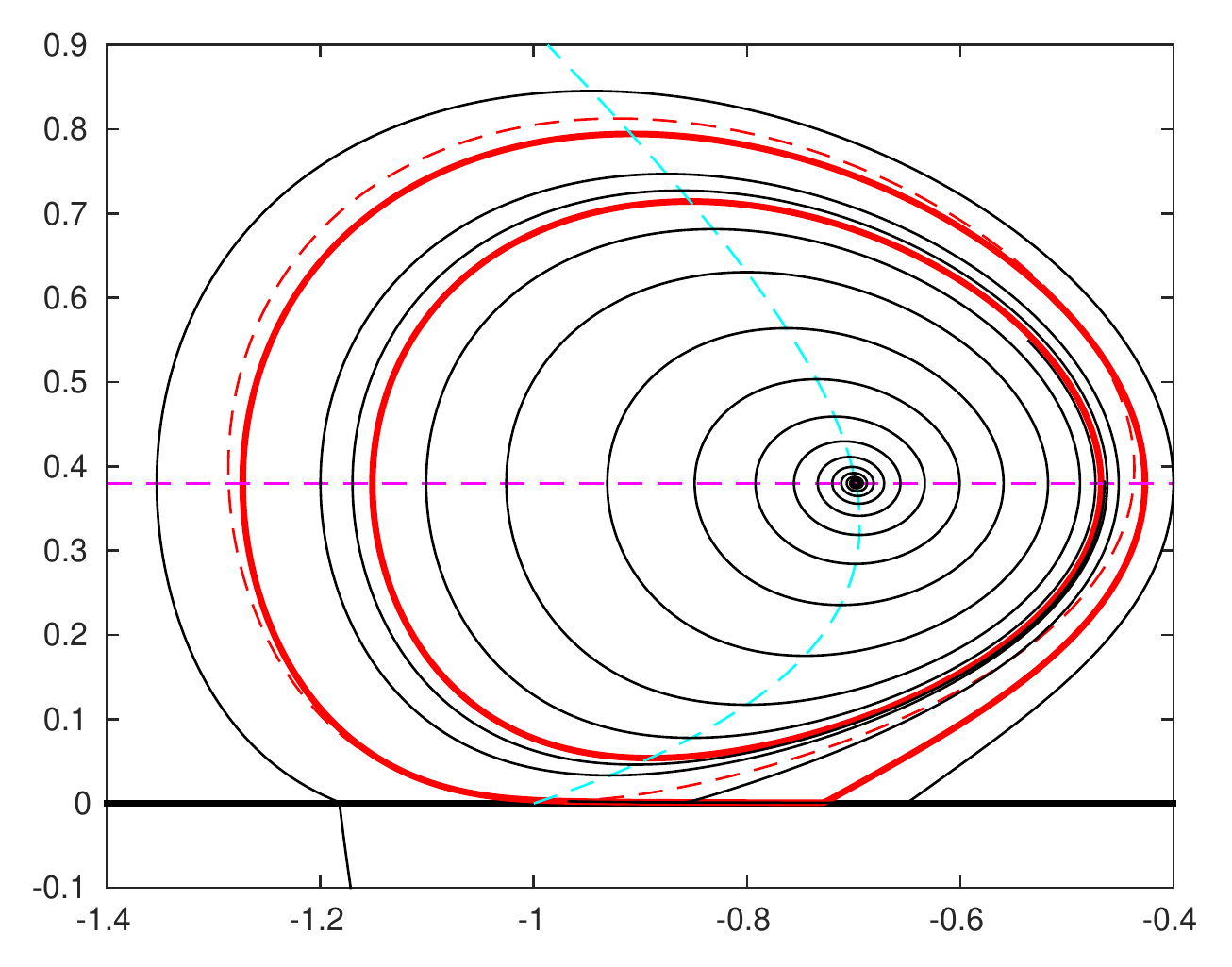}}
 \end{center}
 \caption{In (a): Bifurcation diagram of limit cycles using $\text{min}\,y$ as a measure of the amplitude for varying values of $\epsilon$: in cyan: $\epsilon = 5\times 10^{-3}$, magenta: $\epsilon=2.5\times 10^{-3}$, blue: $\epsilon=10^{-3}$, red: $\epsilon=5\times 10^{-4}$, and finally in green: $\epsilon=10^{-4}$.  In (b): $\alpha_*-\alpha$ along the saddle-node bifurcation for varying values of $\epsilon$. The slope is nearly constant $\approx 0.8024$,  in good agreement with the theoretical value of $4/5$ obtained from \thmref{main2} with $k=2$. In (c): the saddle-node periodic orbits. The dashed magenta and cyan curves are nullclines for $Z_+$ at the unperturbed bifurcation parameter $\alpha=0.4$. The colours are identical to (a). In (d), for $\epsilon=5\times 10^{-3}$, two limit cycles are shown for $\alpha=0.38$. The inner most is repelling while the other one, having a segment near the sliding region, is stable. The black curves are transients while the dashed magenta and cyan curves are nullclines as in (b). For comparison, the saddle-node periodic orbit (red and dashed) is shown for the same $\epsilon$-value.}
  \figlab{auto}
              \end{figure}

\section{Proof of \thmref{main1}}\seclab{proof1}

\qs{In this section, we will prove \thmref{main1}. First, we work with the blowup \eqref{cyl_blowup_map1}. 
The analysis of this blowup system is standard and the details can be found in different formulations, also for more general systems. See e.g. \cite{Buzzi06,Llibre07,kristiansen2018a}. We therefore delay the details to \appref{appA} and instead just summarise the findings (see also \figref{visible1} for an illustration): Using the chart $(\bar \epsilon=1)_2$, recall \eqref{r2y2}, we find a critical manifold $\overline S$ on the cylinder as a graph over $\Sigma_{sl}$. It is noncompact in the scaling chart $(\bar \epsilon=1)_2$, but using $(\bar y=1)_1$ we find that it ends on the edge $\bar y=1$ (yellow  in \figref{visible1}) in a nonhyperbolic point $\overline T:\,x=0,(\bar y,\bar \epsilon)=(1,0)$. This point (in brown) is the imprint of the tangency $T$ (also in brown on the blown down picture on the left) on the blown-up system. Away from $x=0$ the edge $\bar y=1$ is hyperbolic, whereas $\bar y=-1$ (purple in \figref{visible1}) is hyperbolic for all $x$. The latter property follows from working in $(\bar y=-1)_3$. Next, by working in $(\bar \epsilon=1)_2$, we obtain the invariant manifold $S_\epsilon$ using Fenichel's theory \cite{fen3} upon restricting $\overline S$ to the compact set $x\in J$. The invariant foliation in \thmref{main1} (a) is also a consequence of Fenichel's theory. However, Fenichel's foliation is only local to $\overline S$ on the cylinder. To extend it beyond the cylinder into $y\ne 0$ uniformly in $\epsilon$ we work near the hyperbolic lines $(\bar y,\bar\epsilon) = (\pm 1,0)$, $x<0$ in the charts $(\bar y=\pm 1)_{1,3}$, respectively. See further details in \appref{appA}. Combining the information proves \thmref{main1} (a).}
\begin{remark}
In \figref{visible1} and the figures that follow we indicate hyperbolic directions by tripple-headed arrows, whereas slow and center directions are indicated by single-headed arrows.
\end{remark}
\qs{
To prove the remaining claims of the theorem, we work in chart $(\bar y=1)_1$ with the coordinates $(r_1,x,\epsilon_1)$ and the local blowup \eqref{r1epsilon1}. In these coordinates, we then blowup the nonhyperbolic point $\overline T$ to a sphere. Using three directional charts, we describe the dynamics on this sphere, see details in \secref{r1} and \secref{eps1}. This analysis is the basis of the subsequent proof of \thmref{main1}(b),(c) and (d), see \secref{main1bc} and \secref{chini}, respectively.
}
\begin{figure}
\begin{center}
\includegraphics[width=.95\textwidth]{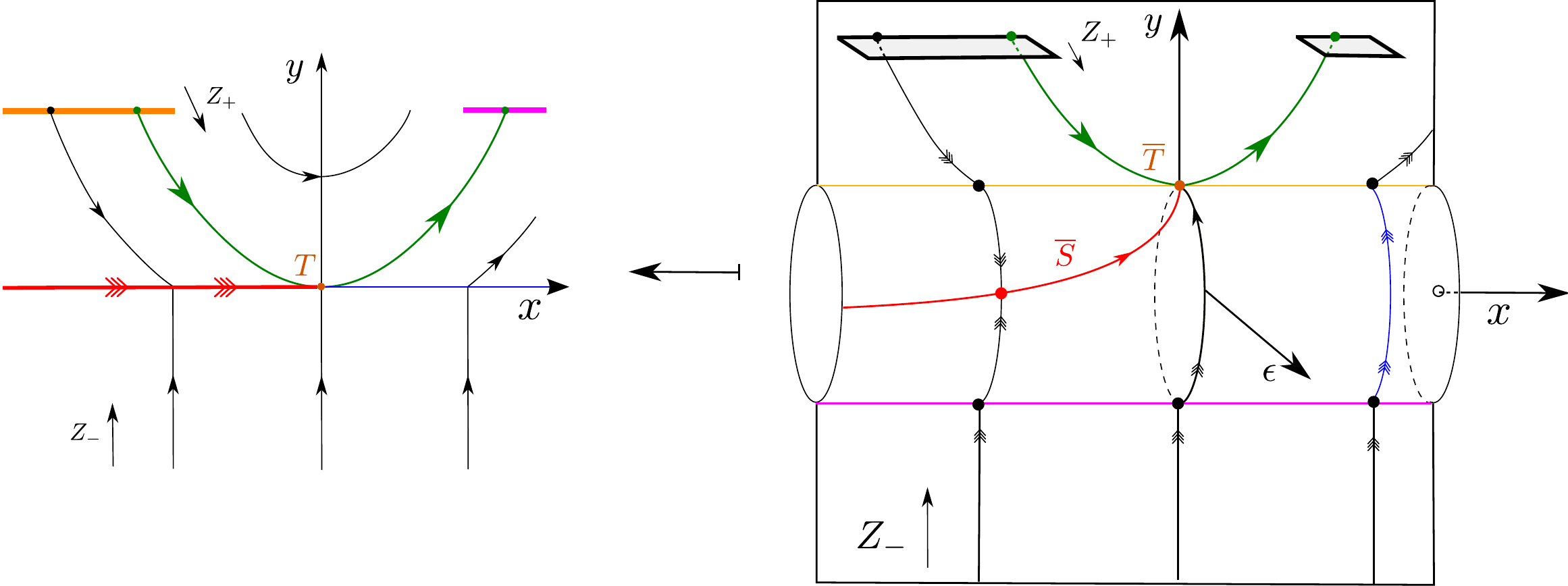}
 \end{center}
\caption{Illustration of the cylindrical blowup of the visible fold. The blown down version is on the left whereas the blowup picture on the right. Our viewpoint is from $\epsilon>0$, this axis coming out of the diagram. \qs{Since $\epsilon\ge 0$ only the part of the cylinder with $\bar \epsilon\ge 0$ is relevant.} Through desingularization we smoothness and hyperbolicity along the edges $\bar y=\pm 1$ (yellow and purple), except at the point $\overline T$ at $x=0$ which is fully nonhyperbolic. On the side of the cylinder, which can be described in the scaling chart $(\bar \epsilon=1)_2$, we find a normally hyperbolic critical manifold $\overline S$. By working in $(\bar y=1)_1$, we realise that it ends in the fully nonhyperbolic point $\overline T$. Here $\overline S$ is tangent to a nonhyperbolic critical fiber at $x=0$.  }
 \figlab{visible1}
              \end{figure}
\subsection{Blowup of the nonhyperbolic point $\overline T$}\seclab{blowupT}
In $(\bar y=1)_1$ we obtain the following equations:
\begin{align}
\dot r &=r F(r,x,\epsilon),\eqlab{y1Eqs}\\
\dot x &= r (1-\epsilon^k \phi_+(r,\epsilon))(1+f(x,r)),\nonumber\\
 \dot \epsilon &=-\epsilon F(r,x,\epsilon),\nonumber
\end{align}
by inserting \eqref{r1epsilon1} into \eqref{zext} with $Z_\pm$ given by \eqref{ZpmNF} and dividing the resulting right hand side by the common factor $\epsilon_1$ (as promised in our description of the blowup approach, see \secref{mainResults}).
For simplicity, we have also dropped the subscripts on $r_1$ and $\epsilon_1$ in \eqref{y1Eqs}.  Furthermore, in \eqref{y1Eqs},
\begin{align*}
 F(r,x,\epsilon) =(1-\epsilon^k\phi_+(r,\epsilon))(2x+r g(x,r))+\epsilon^k \phi_+(r,\epsilon),
\end{align*}
 where we have used (A2) and set $k_+=k$. This system is described in further details in \appref{appA} (for $x<0$).

 Clearly, $(r,x,\epsilon)=(0,0,0)$, corresponding to $\overline T$, is fully nonhyperbolic for \eqref{y1Eqs}, the linearization having only zero eigenvalues. 
Therefore we blowup this nonhyperbolic point by a $k$-dependent blowup transformation $\Psi$ defined by:
\begin{align}
 \rho\ge 0,\,(\bar r,\bar x,\bar \epsilon)\in S^2\mapsto \begin{cases}r=\rho^{2k}\bar r,\\
 x=\rho^k \bar x,\\
 \epsilon = \rho \bar \epsilon. 
 \end{cases}\eqlab{blowupSphere}
\end{align}
\qss{
Let $X$ denote the right hand side in \eqref{y1Eqs}. Then the exponents (or weights) $2k$, $k$, and $1$ on $\rho$ in the expressions in \eqref{blowupSphere} are so that the vector-field $\overline X=\Psi_* X$ on $(\rho,(\bar r,\bar x,\bar \epsilon))\in [0,\rho_0)\times S^2$, for $\rho_0>0$ sufficiently small, has $\rho^k$ as a common factor. We therefore desingularize by dividing out this common factor and study the vector-field
$\widehat X:=\rho^{-k} \overline X$, being topologically equivalent to $\overline X$ on $\rho>0$, instead. However, since $\widehat X\neq 0$ for $\rho=0$ it will have improved hyperbolicity properties. This is the general idea of blowup, see e.g. \cite{dumortier_1996}. }
\begin{remark}
 Notice that the weights in the expressions for $x$ and $\epsilon$ in \eqref{blowupSphere} are so that on the cylinder $\{r=0\}$, the $k$th-order tangency between the critical manifold, of the form $x=\epsilon_1^k m(\epsilon_1)$ in chart $(\bar y=1)_1$, and the nonhyperbolic critical fiber, at $x=\epsilon_1=0$, gets geometrically separated on the blowup sphere. Similarly, the weights on $x$ and $y=r$ are so that the quadratic tangency within $\{\epsilon=0\}$, due to the visible fold, also gets separated.
\end{remark}

We will use three local charts, obtained by setting $\bar r=1$, $\bar \epsilon=1$ and $\bar x=-1$, to describe this blowup:
\begin{align}
 (\bar r=1)_1:\,\rho_1\ge 0,\,x_1\in \mathbb R,\,\epsilon_1\ge 0\,&\mapsto \left\{\begin{array}{ccc}
                                                                    r &=& \rho_1^{2k},\\
                                                                    x &=& \rho_1^k x_1,\\
                                                                    \epsilon &=&\rho_1 \epsilon_1,
                                                                   \end{array} \right.\eqlab{epsilon11}\\
                                                                    (\bar \epsilon=1)_2:\, \rho_2\ge 0,\,r_2\ge 0,\,x_2\in \mathbb R\,&\mapsto \left\{\begin{array}{ccc}
                                                                    r &=& \rho_2^{2k}r_2,\\
                                                                    x &=& \rho_2^k x_2,\\
                                                                    \epsilon &=&\rho_2,
                                                                   \end{array} \right. \eqlab{epsilon12}\\
                                                                    (\bar x=-1)_3:\,\rho_3\ge 0,\,r_3\ge 0,\,\epsilon_3\ge 0\,&\mapsto \left\{\begin{array}{ccc}
                                                                    r &=& \rho_3^{2k}r_3,\\
                                                                    x &=& -\rho_3^k,\\
                                                                    \epsilon &=&\rho_3\epsilon_3,
                                                                   \end{array} \right.\nonumber
\end{align}
focusing primarily on the two former charts.
As indicated, these charts are enumerated as $(\bar r=1)_1,(\bar \epsilon=1)_2$ and $(\bar x=-1)_3$, respectively. \qs{The coordinate changes between the charts follow from the expressions:
\begin{align*}
 \begin{cases}
  \rho_2 &=\rho_1\epsilon_1,\, \quad \quad \quad r_2= \epsilon_1^{-2k},\,\quad \,\,\,x_2 = \epsilon_1^{-k}x_1,\quad \quad \quad \,\,\,\,\,\text{for $\epsilon_1>0$},\\
  \rho_3 &=\rho_1(-x_1)^{1/k},\, r_2= (-x_1)^{-2},\,\epsilon_3 =(-x_1)^{-1/k}\epsilon_1,\quad \text{for $x_1<0$}.
 \end{cases}
\end{align*}
We will frequently use these coordinate changes without further reference. }
We illustrate the blowup in \figref{visible2}, using a similar viewpoint as in \figref{visible1}. We analyse each of the charts in the following. \qs{In \secref{main1bc}, we combine the results into a proof of \thmref{main1} (b) and (c).}
\begin{figure}
\begin{center}
\includegraphics[width=.99\textwidth]{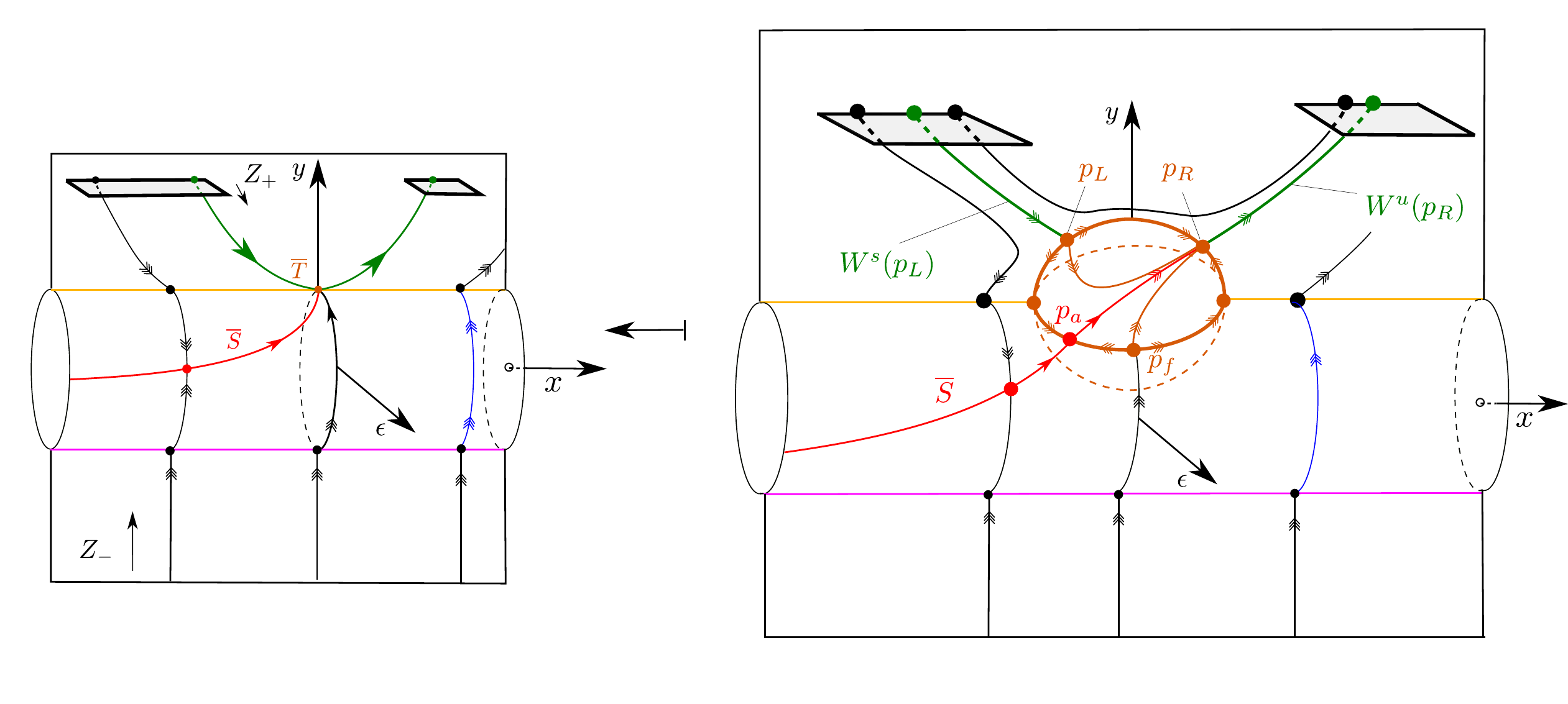} \end{center}
 \caption{Illustration of the subsequent blowup \eqref{blowupSphere} of the nonhyperbolic point $(r_1,x_1,\epsilon_1)=(0,0,0)$, corresponding to $\overline T$, in the chart $(\bar y=1)_1$.  The weights in \eqref{blowupSphere} are so that the critical manifold and the nonhyperbolic fiber gets separated into two poins $p_a$ and $p_f$ on the sphere (in brown in the blowup picture on the right). By desingularization, these points have improved hyperbolicity properties, as indicated by the tripple-headed arrows. Similarly, the blowup also separates the quadratic tangency within $\epsilon=0$ into two points $p_L$ and $p_R$ on the sphere, which also have improved hyperbolicity properties after desingularization (in fact both are fully hyperbolic). See \lemmaref{pLpR} and \lemmaref{M1} for further details. }
\figlab{visible2}
              \end{figure}
              
\subsection{Chart $(\bar r=1)_1$}\seclab{r1}
In this chart, by inserting \eqref{epsilon11} into \eqref{y1Eqs}, we obtain the following equations:
\begin{align}
 \dot \rho_1 &=\frac{1}{2k}\rho_1 F_1(\rho_1,x_1,\epsilon_1),\eqlab{r1Eqns}\\
 \dot x_1 &=  (1-\rho_1^k\epsilon_1^k \phi_+(\rho_1^{2k},\rho_1 \epsilon_1))(1+\rho_1^k f_1(\rho_1^k,x_1))-\frac12 F_1(\rho_1,x_1,\epsilon_1)x_1,\nonumber\\
 \dot \epsilon_1 &= -\frac{2k+1}{2k}F_1(\rho_1,x_1,\epsilon_1),\nonumber
 \end{align}
where 
\begin{align*}
 F_1(\rho_1,x_1,\epsilon_1)=(1-\rho_1^k \epsilon_1^k\phi_+(\rho_1^{2k},\rho_1 \epsilon_1))(2x_1+\rho_1^k g_1(\rho_1^k,x_1))+\epsilon_1^k \phi_+(\rho_1^{2k},\rho_1 \epsilon_1),
\end{align*}
and
\begin{align*}
f_1(\rho_1^k,x_1) &= \rho_1^{-k} f(\rho_1^k x_1,\rho_1^{2k}),\quad g_1(\rho_1^k,x_1) = g(\rho^{k}x_1,\rho_1^{2k}).
\end{align*}
Notice that $f_1$ is well-defined and smooth since $f(0,0)=0$. By \eqref{r1epsilon1} and \eqref{epsilon11}, we have 
\begin{align}
y=\rho_1^{2k},\eqlab{yRho1}
\end{align}
in this chart. Also, $\rho_1=\epsilon_1=0$ is invariant. Along this axis we have
\begin{align*}
 \dot x_1 &=1-x_1^2.
\end{align*}
Therefore $x_1=\mp 1$ are equilibria of this reduced system, $x_1=-1$ being hyperbolic and repelling, $x_1=1$ being hyperbolic and attracting. \qs{They correspond to the intersection of $\gamma$ with the blowup sphere, see \figref{visible2}. }
\begin{lemma}\lemmalab{pLpR}
 The points $p_{L},p_R:\,(\rho_1,x_1,\epsilon_1)=(0,\mp 1,0)$, respectively, are hyperbolic. In particular, the eigenvalues of $p_{L}$ and $p_R$ are as follows:
 \begin{align*}
  \text{for $p_L$}&:\,\lambda_1=-\frac{1}{k},\,\lambda_2 = 2,\,\lambda_3 = 2+\frac{1}{k},\\
  \text{for $p_R$}&:\,\lambda_1=-2-\frac{1}{k},\,\lambda_2 = -2,\,\lambda_3 = \frac{1}{k}.
 \end{align*}
Moreover, the $2$-dimensional $W^u_{loc}(p_L)$ is a neighborhood of $x_1=-1$, $\epsilon_1=0$ within $\rho_1=0$ whereas the $1$-dimensional $W^{s}(p_L)$ -- corresponding to $\gamma$ for $x<0$, $\epsilon=0$ upon blowing down -- is tangent to the vector $(1,0,0)$. On the other hand, the $2$-dimensional $W^s_{loc}(p_R)$ is a full neighborhood of $x_1=1$, $\epsilon_1=0$ within $\rho_1=0$ whereas the $1$-dimensional $W^u(p_R)$ -- corresponding to $\gamma$ for $x>0$, $\epsilon=0$ upon blowing down -- is tangent to the vector $(1,0,0)$.  
\end{lemma}
\begin{proof}
 Calculation. 
\end{proof}

\begin{figure}
\begin{center}
\includegraphics[width=.75\textwidth]{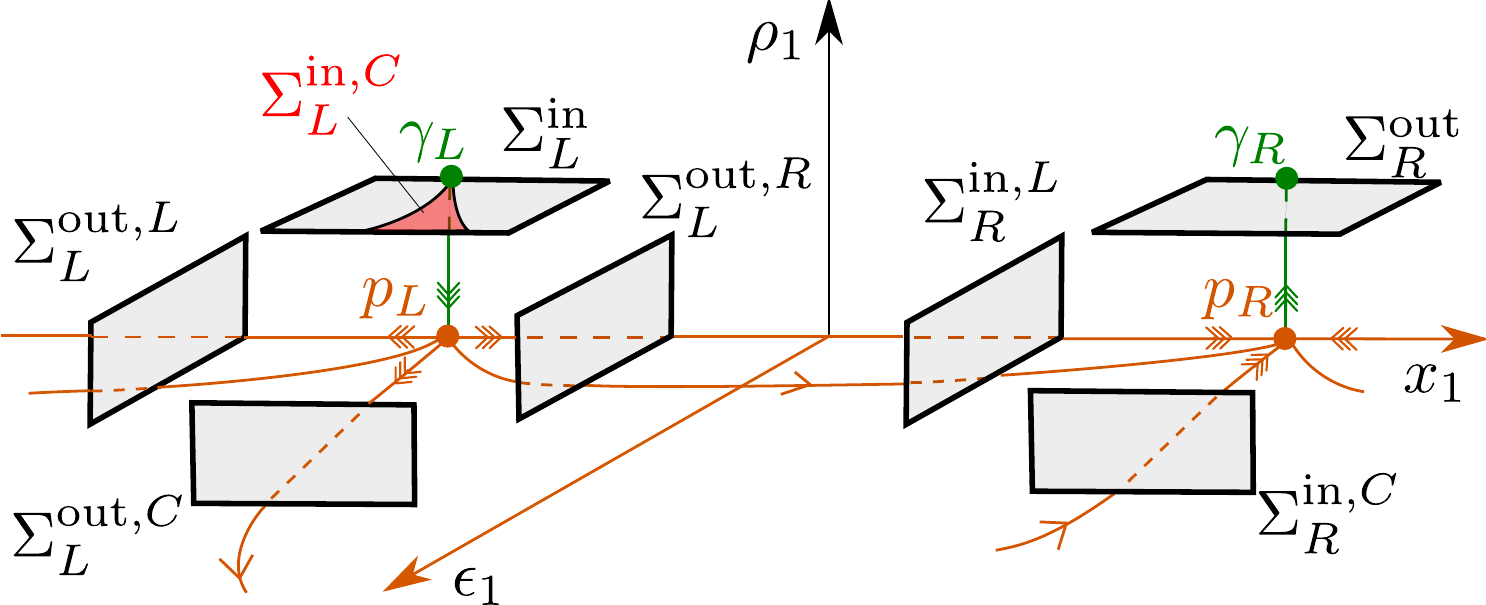}\end{center}
 \caption{Illustration of the dynamics in chart $(\bar r=1)_1$, see \eqref{epsilon11}. $p_L$ and $p_R$ are hyperbolic and we use partial, smooth linearizations, see \lemmaref{tilderho1tildex1Lemma} and \lemmaref{rho1x1eps1Tilde}, near this points to describe the local transition maps between the various sections shown. Notice $\rho_1=0$ corresponds to the sphere, obtained from the blowup of $\overline T$, which is (also) brown in \figref{visible2}.}
  \figlab{barRhoEq1}
              \end{figure}
         See \figref{barRhoEq1} for an illustration, compare also with the sphere on the right in \figref{visible2}.   
         
         \subsubsection*{\qss{A transition map $P_L^C$ near $p_L$}}
         \qs{ For the proof of \thmref{main1} we will need detailed information about transition maps near $p_{L/R}$.}  However, notice that there are strong resonances at $p_{L/R}$: 
\begin{align}
 \lambda_2 -\lambda_1-\lambda_3 &= 0,\eqlab{strongres}
 \end{align}
and hence we cannot (directly, at least) perform  a smooth linearization near these points. However, near $p_L$ and $p_R$ we have $F_1\approx \mp 2$, respectively, and we can therefore divide the right hand side of the equations by $-\frac12 F_1$ and $\frac12 F_1$, respectively, in a neighborhood of these points. Near $p_L$, for example, this produces
\begin{align}
 \dot \rho_1 &= -\frac{1}{k}\rho_1,\eqlab{pMinusEqn}\\
 \dot x_1 &= x_1-\frac{2(1+\rho_1^k f_1(\rho_1^k,x_1)}{2x_1+\rho_1^k g_1(\rho_1^k,x_1)}+\epsilon_1^k G_1(\rho_1,x_1,\epsilon_1),\nonumber\\
 \dot \epsilon_1 &=\frac{2k+1}{k}\epsilon_1,\nonumber
\end{align}
for some smooth $G_1$. \qs{We will then proceed (as is standard, see e.g. \cite{Gucwa2009783,kosiuk2011a,kristiansen2018a} and many others in similar contexts) to apply ``partial linearizations'' within invariant subsets. In particular, within $\epsilon_1=0$, where
\begin{align}
 \dot \rho_1 &=-\frac{1}{k}\rho_1,\eqlab{epsilon1Eq0Eqs}\\
 \dot x_1 &=x_1-\frac{2(1+\rho_1^k f_1(\rho_1^k,x_1)}{2x_1+\rho_1^k g_1(\rho_1^k,x_1)},\nonumber
\end{align}
we find a smooth linearization by exploiting its connection to $Z_+$, recall \eqref{ZpmNF}. Firstly, we have.}
\begin{lemma}\lemmalab{tilderho1tildex1Lemma}
 There exists a diffeomorphism defined by 
 \begin{align}
  (\tilde \rho_1,\tilde x_1)\mapsto \left\{\begin{array}{ccc}
                                            \rho_1& =& \tilde \rho_1 \widetilde{\mathcal R}_L^1(\tilde \rho_1^k,\tilde x_1),\\
                                            x_1  &=& \tilde x_1 \widetilde{\mathcal X}_L^1(\tilde \rho_1^k,\tilde x_1),
                                           \end{array}\right.\eqlab{tilderho1tildex1}
 \end{align}
 where 
 \begin{align}
   \widetilde{\mathcal R}_L^1(\tilde \rho_1^k,\tilde x_1),\widetilde{\mathcal X}_L^1(\tilde \rho_1^k,\tilde x_1)=1+\mathcal O(\tilde \rho_1^k),\eqlab{Q1prop}
 \end{align}
 for $\tilde x_1\in I$, $I$ a fixed open large interval, and $\tilde \rho_1\in [0,\xi]$. Furthermore, there exists a smooth and positive function $T$ -- defined on the same set and satisfying $T(0,\tilde x_1)=1$ for all $\tilde x_1$ -- such that upon applying \eqref{tilderho1tildex1} to \eqref{epsilon1Eq0Eqs}, we have
\begin{align*}
 \dot{\tilde \rho}_1 &=-\frac{1}{k}\tilde \rho_1 T(\tilde \rho_1,\tilde x_1),\\
 \dot{\tilde x}_1&= \left(\tilde x_1-\frac{1}{\tilde x_1}\right) T(\tilde \rho_1,\tilde x_1).
\end{align*}
in a neighborhood of $(\tilde \rho_1,\tilde x_1)=(0,-1)$. 
\end{lemma}
\begin{proof}
By the flow-box theorem there exists a smooth, local diffeomorphism conjugating $Z_+$:
\begin{align*}
 \dot x &=1+f(x,y),\\
 \dot y &=2x+yg(x,y),
\end{align*}
with
\begin{align*}
 \dot{\tilde x} &=1,\\
 \dot{\tilde y} &=2\tilde x, 
\end{align*}
of the form
\begin{align}
 (\tilde x,\tilde y)\mapsto \left\{\begin{array}{ccc}
                                            x& =& \tilde x \widetilde{\mathcal X}(\tilde x,\tilde y),\\
                                            y  &=& \tilde y +\widetilde{\mathcal Y}(\tilde x,\tilde y),
                                           \end{array}\right.\eqlab{1sttransf}
\end{align}
where $\widetilde{\mathcal X}(0,0)=1$ and $\widetilde{\mathcal Y}(x,y)=\mathcal O(2)$ (i.e. $\widetilde{\mathcal Y}(0,0)=0$, $D\widetilde{\mathcal Y}(0,0)=0$). Notice that the transformation fixes the first axis. Furthermore, calculations also  show that $\widetilde{\mathcal Y}(x,y)-g(0,0)xy=\mathcal O(3)$. Now, we define $\tilde \rho_1$ and $\tilde x_1$ by $\tilde y=\tilde \rho_1^{2k}$, $\tilde x=\tilde \rho_1^k \tilde x_1$. Inserting this into \eqref{1sttransf} using $y=\rho_1^{2k}$ and $x=\rho_1^kx_1$ produces \eqref{tilderho1tildex1}. A straightforward calculation verifies the property in the lemma.
\end{proof}
Next, we define $\tilde \epsilon_1$ by
\begin{align}
 \epsilon_1 &=\tilde \epsilon_1 \widetilde{\mathcal R}_L^1(\tilde \rho_1^k,\tilde x_1)^{-2k-1},\eqlab{TILDEEPS}
\end{align}
which is invertible locally by \eqref{Q1prop}. 
Recall that $\rho_1^{2k+1}\epsilon_1= \epsilon$ const. Therefore by construction 
$$\tilde \rho_1^{2k+1} \tilde \epsilon_1 =\epsilon,$$
by \eqref{tilderho1tildex1} and \eqref{TILDEEPS}, also in the new tilde-coordinates $(\tilde \rho_1,x_1,\tilde \epsilon_1)$. In total: 
\begin{lemma}\lemmalab{rho1x1eps1Tilde}
The diffeomorphism defined by
\begin{align}
 (\tilde \rho_1,\tilde x_1,\tilde \epsilon_1)\mapsto  \left\{\begin{array}{ccc}
                                            \rho_1& =& \tilde \rho_1 \widetilde{\mathcal R}_L^1(\tilde \rho_1^k,\tilde x_1),\\
                                            x_1  &=& \tilde x_1 \widetilde{\mathcal X}_L^1(\tilde \rho_1^k,\tilde x_1),\\
                                            \epsilon_1&=&\tilde \epsilon_1 \widetilde{\mathcal R}_L^1(\tilde \rho_1^k,\tilde x_1)^{-2k-1},
                                           \end{array}\right.\eqlab{totaldiffeo}
\end{align}
transforms \eqref{pMinusEqn} into 
\begin{align}
 \dot{\tilde \rho}_1 &= -\frac{1}{k}\tilde \rho_1 T(\tilde \rho_1,\tilde x_1),\eqlab{tilderho1Eqss}\\
 \dot{\tilde x}_1 &=  \left(\tilde x_1-\frac{1}{\tilde x_1}\right) T(\tilde \rho_1,\tilde x_1) + \tilde \epsilon_1^k \widetilde G_1(\tilde \rho_1,\tilde x_1,\tilde \epsilon_1),\nonumber\\
 \dot{\tilde \epsilon}_1&=\frac{2k+1}{k}\tilde \epsilon_1 T(\tilde \rho_1,\tilde x_1),\nonumber
\end{align}
for which $\tilde \rho_1^{2k+1} \tilde \epsilon_1=\epsilon$ is a conserved quantity. 

\end{lemma}
We now drop the tildes on $(\tilde \rho_1,\tilde x_1,\tilde \epsilon_1)$ and transform time by dividing the right hand side of \eqref{tilderho1Eqss} by $T$. This produces the following equations
\begin{align*}
 \dot{\rho}_1 &= -\frac{1}{k} \rho_1,\\
 \dot{x}_1 &=  x_1-\frac{1}{x_1} + \epsilon_1^k G_1( \rho_1,x_1,\epsilon_1),\\
 \dot{\epsilon}_1&=\frac{2k+1}{k}\epsilon_1.
\end{align*}
Next, for the $\rho_1=0$ sub-system:
\begin{align}
 \dot{x}_1 &=  x_1-\frac{1}{x_1} + \epsilon_1^k G_1( 0,x_1,\epsilon_1),\eqlab{system11}\\
 \dot{\epsilon}_1&=\frac{2k+1}{k}\epsilon_1,\nonumber
\end{align}
the linearization about $x_1=-1$, $\epsilon_1=0$ produces eigenvalues $\lambda_2$ and $\lambda_3$ which are nonresonant. Therefore there exists a smooth local diffeomorphism defined by
\begin{align}
 (\tilde x_1,\tilde \epsilon_1)\mapsto \left\{\begin{array}{ccc}
                                               x_1 &=&\widetilde{\mathcal X}_L^2(\tilde x_1,\tilde \epsilon_1)\\
                                               \epsilon_1 &=&\tilde \epsilon_1,
                                              \end{array}\right.\eqlab{linearize}
\end{align}
with $\widetilde{\mathcal X}_L^2(0,0)=-1$, $D\widetilde{\mathcal X}_L^2(0,0)=(1,*)$, 
that linearizes \eqref{system11}. Here $*$ is an unspecified entry, which -- following \eqref{system11} -- is $0$ for any $k\ge 2$. Applying the transformation \eqref{linearize} to the full system (fixing $\rho_1$) produces
\begin{align}
 \dot{\rho}_1 &= -\frac{1}{k} \rho_1,\eqlab{finalL}\\
 \dot{x}_1 &=  2x_1+ \epsilon_1^k G_1( \rho_1,x_1,\epsilon_1),\nonumber\\
 \dot{\epsilon}_1&=\frac{2k+1}{k}\epsilon_1,\nonumber
\end{align}
for some new smooth $G_1( \rho_1,x_1,\epsilon_1)=\mathcal O(\rho_1 \epsilon_1+\rho_1^{k})$, using again the same symbols for simplicity.  We illustrate the local dynamics in \figref{barRhoEq1}. 

\qss{
Denote the resulting diffeomorphism, obtained by composing \eqref{totaldiffeo} with \eqref{linearize}, by $\Psi_L$. It takes the following form:
\begin{align}
 (\tilde \rho_1,\tilde x_1,\tilde \epsilon_1)\mapsto  \left\{\begin{array}{ccc}
                                            \rho_1& =& \tilde \rho_1 \widetilde{\mathcal R}_L(\tilde \rho_1^k,\tilde x_1,\tilde \epsilon_1),\\
                                            x_1  &=&  \widetilde{\mathcal X}_L(\tilde \rho_1^k,\tilde x_1,\tilde \epsilon_1),\\
                                            \epsilon_1&=&\tilde \epsilon_1 \widetilde{\mathcal R}_L(\tilde \rho_1^k,\tilde x_1,\tilde \epsilon_1)^{-2k-1},
                                           \end{array}\right.\eqlab{totaldiffeo1}
\end{align}
with smooth functions satisfying $\widetilde{\mathcal R}_L(\tilde \rho_1^k,\tilde x_1,\tilde \epsilon_1) = \mathcal O(\tilde \rho_1^k)$, $\widetilde{\mathcal X}_L(0,0,0) = -1$, $D\widetilde{\mathcal X}_L(0,0,0) = (0,1,*)$. 
In particular, the one-dimensional mapping
\begin{align}
 \tilde x_1\mapsto x_1=\widetilde{\mathcal X}_L(\tilde \rho_1^k,\tilde x_1,\tilde \epsilon_1),\eqlab{tildex1x}
\end{align}
obtained from the $x_1$-entry of \eqref{totaldiffeo1} by fixing any $\tilde \rho_1\in [0,\xi]$ and $\tilde \epsilon_1\in [0,\xi]$, is a diffeomorphism on the set defined by $\tilde x_1\in [-\xi,\xi]$, taking $\xi$ small enough.}

\qss{
We now consider the following section,
$$\Sigma_L^{\text{in}} =\left\{(\rho_1,x_1,\epsilon_1)\vert \rho_1=\delta^{1/2k}>0,\,x_1 \in [-\beta_1,\beta_1],\,\epsilon_1 \in [0,\beta_2]\right\},$$ instead of $\Sigma_L$, recall \eqref{SigmaL0} and \eqref{yRho1}, containing $\gamma_L$ as $\rho_1=\delta^{1/2k},x_1=0,\epsilon_1=0$ in these coordinates. Obviously, $\delta^{1/2k},\beta_1,\beta_2$ are all less than $\xi$. }
\begin{lemma}\lemmalab{X1C}
\qss{
For any $\eta>0$, $\delta>0$ and $\nu>0$ small enough, we 
let $\theta\in (0,\eta)$ and consider the wedge
 \begin{align}
  \Sigma_L^{{\text{in}},C} = \Sigma_L^{\text{in}} \cap \{\vert x_1 \vert \le \theta (\epsilon_1\nu^{-1})^{2k/(2k+1)}\},\eqlab{sigmaLinC}
 \end{align}}
 consisting of all points in $\Sigma_L^{{\text{in}}}$ with $\vert x_1 \vert \le \theta (\epsilon_1\nu^{-1})^{2k/(2k+1)}$, and the section 
\begin{align*}
 \Sigma_L^{\text{out},C} = \{(\rho_1,x_1,\epsilon_1)\vert \rho_1\in [0,\beta_3],\,x_1 \in [-\eta,\eta],\,\epsilon_1=\nu\}.
\end{align*}
(The sections $\Sigma_L^{{\text{in}},C}$ and $\Sigma_L^{{\text{out}},C}$ are illustrated in \figref{barRhoEq1} in the original coordinates.)
Then there exist appropriate constants $\beta_i$, $i=1,2,3$ such that the transition map $P_L^{C}:\Sigma_L^{\text{in},C}\rightarrow \Sigma_L^{\text{out},C}$ obtained by the forward flow of \eqref{finalL} is well-defined and of the following form
\begin{align*}
 P_L^C(\rho_1,x_1,\epsilon_1) = \begin{pmatrix}
                               \left(\epsilon_1 \nu^{-1}\right)^{1/(2k+1)} \delta^{1/2k}\\
                               X_L^C(x_1,\epsilon_1)\\
                               \nu
                              \end{pmatrix}
\end{align*}
where $X_L^C(\cdot,\epsilon_1)$  is $C^2$ $\mathcal O(\epsilon_1^{1/(2k+1)})$-close to the linear map $x_1\mapsto \left(\epsilon_1\nu^{-1}\right)^{-2k/(2k+1)} x_1$:
\begin{align}
 X_L^C(x_1,\epsilon_1)=\left(\epsilon_1\nu^{-1}\right)^{-2k/(2k+1)} x_1+\mathcal O(\epsilon_1^{1/(2k+1)}).\eqlab{XLC}
\end{align}
\end{lemma}
\begin{proof}
 We consider \eqref{finalL} and integrate the linear $\rho_1$- and $\epsilon_1$-equations and insert the results into the $x_1$-equation. We then write $x_1=e^{2t} u$ and estimate $u$ through direct integration. Returning to $x_1$ gives the desired result. The derivatives of $X_L^C$ with respect to $x_1$ can be handled in the exact same way by looking at the variational equations. The estimates on $u$ do not change by this differentiation. 
 \end{proof}
  \begin{remark}\remlab{theta}
  \qss{
   Notice that $$P_L^C(\rho_1,\pm \theta (\epsilon_1\nu^{-1})^{2k/(2k+1)},\epsilon_1)=\begin{pmatrix}
                               \left(\epsilon_1 \nu^{-1}\right)^{1/(2k+1)} \delta^{1/2k}\\
                               \pm \theta+\mathcal O(\epsilon_1^{1/(2k+1)})\\
                               \nu
                              \end{pmatrix},$$
                              and the parameter $\theta\in (0,\eta)$ in \eqref{sigmaLinC} therefore measures the extent to which $\Sigma_L^{{\text{in}},C}$ ``stretches upon reaching $\Sigma_L^{\text{out},C}$ through the forward flow. }
  \end{remark}

\subsubsection*{\qss{A transition map $P_R^C$ near $p_R$}}
Returning to \eqref{r1Eqns}, we can perform the exact same analysis near $p_R$. In other words: Near $p_R$ there exists a regular transformation of time and a diffeomorphism $\Psi_R$ -- defined for $\tilde \rho_1,\tilde x_1,\tilde \epsilon_1\in [0,\xi]$ with $\xi$ small enough -- of the following form
\begin{align}
 (\tilde \rho_1,\tilde x_1,\tilde \epsilon_1)\mapsto  \left\{\begin{array}{ccc}
                                            \rho_1& =& \tilde \rho_1 \widetilde{\mathcal R}_R(\tilde \rho_1^k,\tilde x_1,\tilde \epsilon_1),\\
                                            x_1  &=& \widetilde{\mathcal X}_R(\tilde \rho_1^k,\tilde x_1,\tilde \epsilon_1),\\
                                            \epsilon_1&=&\tilde \epsilon_1 \widetilde{\mathcal R}_R(\tilde \rho_1^k,\tilde x_1,\tilde \epsilon_1)^{-2k-1},
                                           \end{array}\right.\eqlab{totaldiffeoPsiR}
\end{align}
%
with smooth functions satisfying $\widetilde{\mathcal R}_R(\tilde \rho_1^k,\tilde x_1,\tilde \epsilon_1)=\mathcal O(\tilde \rho_1^k)$, $\widetilde{\mathcal X}_R(0,0,0) = 1$, and $D\widetilde{\mathcal X}_R(0,0,0) = (0,1,*)$, that together transform \eqref{r1Eqns} into
%
\begin{align}
 \dot \rho_1 &= \frac{1}{k}\rho_1,\eqlab{here}\\
 \dot x_1 &=-2x_1+\epsilon_1^k G_1(\rho_1,x_1,\epsilon_1),\nonumber\\
 \dot \epsilon_1 &=-\frac{2k+1}{k}\epsilon_1,\nonumber
\end{align}
for which $\rho_1^{2k+1}\epsilon_1=\epsilon$ is conserved. Here we have dropped the tildes on $(\tilde \rho_1^k,\tilde x_1,\tilde \epsilon_1)$ and introduced  $G_1(\rho_1,x_1,\epsilon_1)=\mathcal O(\epsilon_1\rho_1+\rho_1^k)$ as a (new) smooth function. 

As for $\Psi_L$, the one-dimensional mapping
\begin{align}
\tilde x_1\mapsto x_1  &= \widetilde{\mathcal X}_R(\tilde \rho_1^k,\tilde x_1,\tilde \epsilon_1),\eqlab{x1PsiR}
\end{align}
obtained from the $x_1$-entry of \eqref{totaldiffeoPsiR} by fixing any $\tilde \rho_1$, $\tilde \epsilon_1\in [0,\xi]$, is also a diffeomorphism on the set defined by $\tilde x_1\in [-\xi,\xi]$. We write the inverse of \eqref{x1PsiR}, which will be important in our analysis below of the mapping $Q$ in \thmref{main1}, as
\begin{align}
x_1\mapsto \tilde x_1  &= {\mathcal X}_R(\tilde \rho_1^k,x_1,\tilde \epsilon_1),\eqlab{x1PsiRinv}
\end{align}
satisfying $\mathcal X_R(0,1,0,)=0$, $D\mathcal X_R(0,1,0)=(0,1,*)$. 



In the coordinates of \eqref{here}, we now consider
\begin{align*}
 \Sigma_R^{\text{out}} =\left\{(\rho_1,x_1,\epsilon_1)\vert \rho_1=\delta^{1/2k}>0,\,x_1 \in [-\beta_1,\beta_1],\,\epsilon_1 \in [0,\beta_2]\right\},
\end{align*}
instead of $\Sigma_R$, recall \eqref{SigmaR0}, 
containing $\gamma_L$ as $\rho_1=\delta^{1/2k},x_1=0,\epsilon_1=0$ in these coordinates.  
\begin{lemma}\lemmalab{P1}
\qss{
For any $\eta>0$, $\delta>0$ and $\nu>0$ small enough, we 
consider the section
$$\Sigma_R^{\text{in},C} = \{(\rho_1,x_1,\epsilon_1)\vert \rho_1\in [0,\beta_3],\,x_1 \in [-\eta,\eta],\,\epsilon_1=\nu\}.$$ }
(The sections $\Sigma_R^{{\text{in}},C}$ and $\Sigma_R^{{\text{out}}}$ are illustrated in \figref{barRhoEq1} in the original coordinates.)
Then there exist appropriate constants $\beta_i$, $i=1,\ldots,3$ such that the transition map $P_R^{C}:\Sigma_R^{\text{in},C}\rightarrow \Sigma_R^{\text{out}}$ obtained by the forward flow is well-defined and of the following form
\begin{align*}
 P_R^C(\rho_1,x_1,\epsilon_1) = \begin{pmatrix}
                               \delta^{1/2k}\\
                               X_{R}^C(\rho_1,x_1)\\
                               \left(\rho_1\delta^{-1/2k}\right)^{2k+1} \nu
                              \end{pmatrix}
\end{align*}
where $X_R^C(\rho_1,\cdot)$ is $C^2$ $\mathcal O(\rho_1^{2k+1})$-close to the linear map $x_1\mapsto \rho_1^{2k} \delta^{-1} x_1$:
\begin{align}
 X_R^C(\rho_1,x_1)=\rho_1^{2k} \delta^{-1} x_1+\mathcal O(\rho_1^{2k+1}).\eqlab{X1C}
\end{align}
\end{lemma}
\begin{proof}
 We consider \eqref{here} and proceed as in the proof of \lemmaref{X1C} by integrating the $\rho_1$- and the $\epsilon_1$-equations and insert the result into the $x_1$-equation. Further details are therefore left out.
\end{proof}

%
%
%

\subsection{Chart $(\bar \epsilon=1)_2$}\seclab{eps1}
In this chart, by inserting \eqref{epsilon12} into \eqref{y1Eqs}, we obtain the following equations
\begin{align*}
 \dot \rho_2 &=-\rho_2 F_2(\rho_2,r_2,x_2),\\
 \dot r_2 &=  (2k+1)r_2 F_2(\rho_2,r_2,x_2),\\
 \dot x_2 &=r_2\left(1-\rho_2^k \phi_+(\rho_2^{2k}r_2,\rho_2)\right)\left(1+\rho_2^kf_2(\rho_2,r_2,x_2)\right)+kx_2F_2(\rho_2,r_2,x_2),
 \end{align*}
where 
\begin{align*}
 F_2(\rho_2,r_2,x_2)=(1-\rho_2^k \phi_+(\rho_2^{2k}r_2,\rho_2))(2x_2+\rho_2^k r_2 g_2(\rho_2^k, r_2,x_2))+ \phi_+(\rho_2^{2k}r_2,\rho_2),
\end{align*}
and
\begin{align*}
f_2(\rho_2^k,r_2,x_2) &= \rho_2^{-k} f(\rho_2^{k}x_2,\rho_2^{2k} r_2),\quad g_2(\rho_2^k,r_2,x_2) = g(\rho_2^{k}x_2,\rho_2^{2k}r_2).
\end{align*}
Along the invariant set $\rho_2=r_2=0$, we have
\begin{align*}
 \dot x_2 = kx_2 \left(2x_2+\phi_+(0,0)\right),
\end{align*}
so that $x_2=0$ and $x_2=-\frac12 \phi_+(0,0)$ are equilibria, the former being repelling while the latter is attracting. Linearization of the full system about these equilibria produces the following lemma by standard theory \cite{perko2001a}.
\begin{lemma}\lemmalab{M1}
We have
\begin{enumerate}
 \item The point $p_a:\,(\rho_2,r_2,x_2)=\left(0,0,-\frac12 \phi_+(0,0)\right)$ is partially hyperbolic, the linearization having only one single non-zero eigenvalue $\lambda = -k\phi_+(0,0)<0$. As a consequence there exists a center manifold $M_{1}$ of $p_a$ which contains $S_1$ within $r_2=0$ as a manifold of equilibria and a unique center manifold within $\rho_2=0$, along which $r_2$ is increasing, which is tangent to the eigenvector $(0,k\phi_+,1)^T$. The equilibrium $p_a$ is therefore a nonhyperbolic saddle.
 \item The point $p_f:\,(\rho_2,r_2,x_2)=(0,0,0)$ is fully hyperbolic, the linearization having two positive eigenvalues and one negative. The stable manifold is $r_2=x_2=0,\,\rho_2\ge 0$ whereas $W^u_{loc}$ is a neighborhood of $(r_2,x_2)=(0,0)$ within $\rho_2=0$. 
 \end{enumerate}
\end{lemma}
\begin{proof}
Calculations.
\end{proof}

\subsection{Proof of \thmref{main1}(b) and (c)}\seclab{main1bc}
In chart $(\bar \epsilon=1)_2$, the set defined by the following equation
\begin{align}
 \rho_2^{2k+1}r_2^{2k} = \epsilon,\eqlab{conservation}
\end{align}
where $\epsilon$ is the original small parameter, is invariant. This follows from \eqref{epsilon12} and \eqref{r1epsilon1}. Therefore if we fix $\epsilon>0$ sufficiently small and restrict to the set defined by \eqref{conservation}, 
the local center manifold $M_1$ in \lemmaref{M1} provides an extension of the invariant manifold $S_\epsilon$ up to $r_2=\upsilon$, $\upsilon $ a {small constant}, in the usual way; see e.g. \cite{krupa_extending_2001}. At $r_2=\upsilon$ we have 
\begin{align*}
 \rho_2 = (\epsilon \upsilon^{-2k})^{1/(2k+1)},
\end{align*}
by \eqref{conservation},
and hence 
\begin{align*}
 y = \rho_2^{2k}r_2= (\epsilon \upsilon^{-2k})^{2k/(2k+1)}\upsilon.
\end{align*}
In fact, following the analysis of the standard, slow-fast, planar, regular fold point in \cite{krupa_extending_2001} we obtain a similar result to \cite[Proposition 2.8]{krupa_extending_2001} for the local transition map from $\rho_2=\nu$ to $r_2=\upsilon$ near $p_a$ that is exponentially contracting like $e^{-cr_2^{-1}}$ with $c>0$.  

Next, since $r_2$ is increasing on $M_1$, we may track the slow manifold across the sphere, using regular perturbation, Poincar\'e-Bendixson and the analysis in $(\bar r=1)_1$ in the previous subsection, up close to $p_R$. We then use \lemmaref{P1} and the mapping $P_R^C$ to describe the passage near $p_R$. Recall that $p_R$ is a stable node on the sphere, attracting every point on the quarter sphere $\bar \epsilon\ge 0$, $\bar r\ge 0$, except for certain subsets of the invariant half-circles $\bar r=0$ and $\bar \epsilon=0$. See \figref{visible2}. By the expression in \eqref{X1C}, and the following conservation 
\begin{align}
 \rho_1 = \left(\epsilon \nu^{-1}\right)^{1/(2k+1)}, \eqlab{conservation1}
\end{align}
in chart $(\bar r=1)_1$ at $\epsilon_1=\nu$ on $\Sigma_R^{\text{in},C}$, obtained by 
combining \eqref{epsilon12} and \eqref{r1epsilon1}, we reach the result on the slow manifold in \thmref{main1}(b). Also, combining the exponential contraction near $p_a$ in chart $(\bar \epsilon=1)_2$ with the algebraic contraction in \lemmaref{P1}, we obtain the contraction of the local map $Q\vert_K$ as detailed in \thmref{main1}(c). 
\subsection{Proof of \thmref{main1}(d)}\seclab{chini}
For the proof of \thmref{main1}(d), we focus on proving the estimates in (ii) and the inequality \eqref{concaveDown}. The estimates (i) and (iii) are simpler and we will only discuss these at the end of this section.

\qss{
The idea for (ii) is to first work in $(\bar r=1)_1$, applying the diffeomorphisms $\Psi_i^{-1}$ near $p_i$, $i=L,R$, respectively. The domain for $x$ in (ii) then allows us to apply $P_L^C$, recall \lemmaref{rho1x1eps1Tilde} and \lemmaref{X1C}, possibly after adjusting the relevant constants. Upon application of $\Psi_L$, this brings us up to $\epsilon_1=\nu$ and $x_1+1\in [-\theta,\theta]$ with $\theta<\eta$. From here we can -- following the analysis in the chart $(\bar \epsilon=1)_2$ -- guide the flow using regular perturbation theory up close to $p_R$, where we can apply (again after possibly adjusting relevant constants, in particular by decreasing $\theta$, recall \remref{theta}) $P_R^C\circ \Phi_R^{-1}$, see \lemmaref{P1}. We let $Q_1:\Sigma_L^{\text{in},C}\rightarrow \Sigma_R$ denote the resulting mapping:
\begin{align}
 Q_1 &= P_R^C \circ \Psi_R^{-1} \circ P_C \circ \Psi_L\circ P_L^C:\Sigma_L \rightarrow \Sigma_R,\eqlab{Q1comp}
\end{align}
with $P_C$ the diffeomorphism obtained from the finite time flow map of \eqref{r1Eqns} from $\Sigma_L^{\text{out},C}$ to $\Sigma_R^{\text{in},C}$, see \figref{barRhoEq1}. }
%
%

\qss{
There are two main issues regarding the proof of \thmref{main1}(d). We describe these in the following:
\begin{enumerate}
\item Clearly, the mapping $Q$ is conjugated to $Q_1$ on $\epsilon=\text{const}$. sections. (The difference between a section $\Sigma_i$ and $\Psi_i(\Sigma_i)$ for $i=L,R$ is regular -- in particular, it can be adjusted for by a simple application of the implicit function theorem -- and consequently, we will therefore ignore such differences in the following.) Our strategy for proving \thmref{main1} (d) is therefore to prove an analogously statement for the mapping $Q_1$ on $\epsilon=$ const. 
However, let (for good measure) $I_{RL}:\Sigma_L^{\text{out},C}\rightarrow \Sigma_R^{\text{in},C}$ be defined by $I_{LR}(\rho_1,x_1,\epsilon_1)=(\rho_1,x_1,\epsilon_1)$. Then a simple calculation, using \lemmaref{X1C} and \lemmaref{P1}, shows that 
\begin{align*}
\left(P_R^C\circ I_{RL}\circ P_L\right)(\rho_1,x_1,\epsilon_1) = \begin{pmatrix}
                                         \rho_1\\
                                         x_1+\mathcal O(\epsilon_1^{1/(2k+1)})\\
                                         \epsilon_1
                                        \end{pmatrix}
                                        \end{align*}
                                        and hence, not surprisingly,
the expansion of $P_L^C$ is precisely compensated by the contraction of $P_R^C$. Consequently, seeing that $\Psi_L$ and $\Psi_R$ are local, regular diffeomorphisms, the contractive properties of $Q_1$ are essentially given by $P_C$. We therefore need an accurate description of $P_C$ to finish the prove \thmref{main1}(d) (ii). We will provide this description in the following.
\item Secondly, we also need to ensure that for any $c>0$, we can pick the appropriate constants $\eta$, $\theta\in (0,\eta)$, $\delta$, and $\nu$ in \lemmaref{X1C} and \lemmaref{P1} such that $Q_x'(x,\epsilon)$ is monotone on the domain of (ii), attaining all values in the interval $(-1+c,c)$. This will be the purpose of the subsequent sections below. 
\end{enumerate}}

%

%
%

\subsubsection*{\qss{The mapping $P_C$}}

To study $P_C$ our arguments will be based upon regular perturbation theory, exploiting the conservation of $\rho_1^{2k+1}\epsilon_1=\epsilon$, and we therefore first consider the invariant $\rho_1=0$-subsystem:
\begin{align}
\dot x_1 &=1-\frac12 \left\{2x_1+\epsilon_1^k \phi_+\right\}x_1,\eqlab{x1eps1}\\
 \dot \epsilon_1 &=-\frac{2k+1}{2k}\left\{2x_1+\epsilon_1^k\phi_+\right\}\epsilon_1.\nonumber
\end{align}
Here and in the following, we will for simplicity write $\phi_+(0,0)$ as $\phi_+$. In \eqref{x1eps1} we have deliberately emphasized the brackets (using  \textit{curly} brackets) appearing in both equations. This allows us to write the equations in a simpler form.
\begin{lemma} \lemmalab{uvtransformation}
 The diffeomorphism defined by
 \begin{align}
  (x_1,\epsilon_1)\mapsto \begin{cases} u &= (\phi_+\epsilon_1^k)^{-1/(2k+1)} x_1,\\
  v&=(\phi_+\epsilon_1^k)^{-2/(2k+1)},\eqlab{uvc}
  \end{cases}
 \end{align}

for $\epsilon_1>0$, brings \eqref{x1eps1} into the following system
 \begin{align}
  \dot u &= v^{1/2}, \eqlab{uv}\\
  \dot v &= v^{1/2} \left(2u+v^{-k}\right).\nonumber
 \end{align}
\end{lemma}
 \begin{proof}
 Simple calculation.
 \end{proof}
To study \eqref{uv} we multiply the right hand side by $v^{-1/2}$:
\begin{align}
\dot u &=1,\eqlab{chini}\\
 \dot v &=2u+v^{-k},\nonumber
\end{align}

\begin{remark}
The equation \eqref{chini}, written as a first order system, is known as a Chini equation, see e.g. \cite{trifonov2011a}. To the best of the author's knowledge, no solution by quadrature is known to exist.  As a result, our analysis of this system is fairly indirect. 
\end{remark}
\begin{remark}
 Upon using \eqref{cyl_blowup_map1} and \eqref{epsilon11}, it is possible to write \eqref{uvc} as follows
 \begin{align*}
x &= (\phi_+\epsilon^k)^{1/(2k+1)}u,\\
y&=(\phi_+\epsilon^k)^{2/(2k+1)}v.
 \end{align*}
This is the ``appropriate scaling'' for 
\begin{align*}
 \dot x &=1,\\
 \dot y&=2x+\epsilon^k y^{-k} \phi_+,
\end{align*}
which is (in some sense) ``the leading order'' of the regularization of \eqref{ZpmNF} using a regularization function satisfying (A2). 
\end{remark}

In the $(u,v)$-coordinates, and for $\rho_1=0$, $\Sigma_L^{\text{out},C}$ and $\Sigma_R^{\text{in},C}$ both become subsets of $v=(\phi_+\nu^k)^{-2/(2k+1)}$, with $2u+v^{-k}<0$ and $2u+v^{-k}>0$ respectively. In the following, we let $\Sigma_L:=\{(u,v)\vert v=(\phi_+\nu^k)^{-2/(2k+1)},\,2u+v^{-k}<0\}$ and $\Sigma_R:=\{(u,v)\vert v=(\phi_+\nu^k)^{-2/(2k+1)},\,2u+v^{-k}>0\}$, for simplicity,  denote these extended sections within $v=(\phi_+\nu^k)^{-2/(2k+1)}$. Given the form of \eqref{chini}, we can therefore write the mapping $P_C$ for $\rho_1=0$ in the $(u,v)$-variables as a one-dimensional mapping
\begin{align*}
 u\mapsto U(u),
\end{align*}
such that $(u,v)\in \Sigma_L$ gets mapped to $(U(u),v)\in \Sigma_R$, see \figref{uvPlane}, 
where 
\begin{align}
 U(u) = T(u)+u,\eqlab{PCDefine}
\end{align}
$T(u)$ being the time of flight (see also \eqref{Tdefine} below for a formal definition). 
\begin{lemma}\lemmalab{PC}
The following holds
 \begin{align}
U'(u)\in (-1,0),\quad U''(u) <0,\eqlab{UC}
 \end{align}
for all $(u,v)\in \Sigma_L$ with $2u+v^{-k}<0$. 
\end{lemma}
\qss{We delay the proof to the end of this \secref{proof1}. Fix $\zeta>0$ small enough. Then upon returning to $x_1$ using \eqref{uvc}, we obtain the following expression for $P_C$ for $\rho_1=0$:
\begin{align*}
 P_C(0,x_1,\nu) = \begin{pmatrix}
                   0\\
                   X_{C,0}(x_1)\\
                   \nu
                  \end{pmatrix}
\end{align*}
where
\begin{align}
 X_{C,0}(x_1)=(\phi_+ \nu^k)^{1/(2k+1)} U((\phi_+\nu^k)^{-1/(2k+1)} x_1),\eqlab{XC0}
\end{align}
for $x_1\in [-1-\zeta,-1+\zeta]$. By \lemmaref{PC} we have the following:
\begin{lemma}\lemmalab{nu0}
For any $c>0$, there exists a $\nu_0$ such that for any $\nu\in (0,\nu_0]$ we have 
\begin{align*}
 X_{C,0}'(x_1) &\in (-1+c,c),\quad 
 X_{C,0}''(x_1)<0,
\end{align*}
for all $x_1\in [-1-\zeta,-1+\zeta]$.
 \end{lemma}}
\begin{proof}
\qss{
Consider the point on $\Sigma_L^{\text{out},C}$ with $x_1=-1+\zeta$. Then upon decreasing $\nu>0$, a simple calculation shows  that $X_{C,0}'(x_1)$ can be made as close to $-1$ as desired. Basically, the invariance of $\epsilon_1=0$, $x_1\in (-1,1)$ leads us to consider the variational equations. This produces an equation of the form $\epsilon_1'(x) = \epsilon_1 x/(x^2-1)$ upon eliminating time. Using that the right hand side is an odd function in $x_1$, we obtain that the linearized map is $x_1\mapsto -x_1$. }

\qss{On the other hand, if we consider a point $x_1=-1-\zeta$, then upon decreasing $\nu>0$ we can follow the flow by a finite time flow map, working in the charts $(\bar x=-1)_3$ and $(\bar \epsilon=1)_2$ successively, up close to the center manifold of $p_a$, recall \figref{barRhoEq1} and \lemmaref{M1} item (1). The consequence of the contraction towards this manifold is that $X'_{C,0}(x_1)\rightarrow 0^-$ as $\nu\rightarrow 0$, see also \secref{main1bc} above. Using \eqref{UC}, we therefore have that for any $c>0$ that there exists a $\nu_0$ such that for any $\nu\in (0,\nu_0]$ we have
\begin{align*}
 X_{C,0}'(x_1) &\in (-1+c,-c),\quad 
 X_{C,0}''(x_1)<0,
\end{align*}
for all $x_1\in [-1-\zeta,-1+\zeta]$, as desired.}
\end{proof}


\qss{By regular perturbation theory (seeing that $P_C$ is obtained from a finite time flow map) we have the following
\begin{align*}
 P_C(\rho_1,x_1,\nu) = \begin{pmatrix}
                        \rho_1\\
                        X_{C}(\rho_1,x_1)\\
                       \nu
                       \end{pmatrix}
\end{align*}
with $X_C$ smooth:
\begin{align}
 X_C(\rho_1,x_1) = X_{C,0}(x_1)+\mathcal O(\rho_1).\eqlab{XCrho}
\end{align}}

\subsubsection*{\qss{Restricting to invariant $\epsilon=$ const. sections}}
\qss{By restricting to the invariant sets defined by $\rho_1^{2k+1}\epsilon_1=\epsilon$, each of the mappings $P_R^C\circ \Psi_R^{-1}$, $P_C$ and $\Psi_L \circ P_L^C$ become one-dimensional. 
To simplify notation, we define the following functions
\begin{align*}
 X_{R,\epsilon}^C(x_1)&:= X_R^C\left(\left(\epsilon \nu^{-1}\right)^{1/(2k+1)},x_1\right),\\
 \mathcal X_{R,\epsilon}(x_1)&:= \mathcal X_R\left(\left(\epsilon \nu^{-1}\right)^{1/(2k+1)},x_1,\nu\right),\\
 X_{C,\epsilon}(x_1)&:=X_C(\left(\epsilon \nu^{-1}\right)^{1/(2k+1)},x_1),\\
\widetilde{\mathcal X}_{L,\epsilon}(x_1)&:= \widetilde{\mathcal X}_L\left(\left(\epsilon \nu^{-1}\right)^{1/(2k+1)}, x_1,\nu\right),\\
X_{L,\epsilon}^C(x_1)&:=X_L^C\left(x_1,\epsilon \delta^{-(2k+1)/2k}\right),
\end{align*}
Recall the definitions of $X_L^C$ and $X_R^C$ in \lemmaref{X1C} and \lemmaref{P1}, respectively, and the definitions of $\widetilde{\mathcal X}_L$ and ${\mathcal X}_R$ in \eqref{tildex1x} and \eqref{x1PsiRinv}.
In this way,
%
\begin{align}
 x_1 &\mapsto \left(X_{R,\epsilon}^C \circ \mathcal X_{R,\epsilon}\right)(x_1),\eqlab{1dmappings1}\\
 x_1&\mapsto X_{C,\epsilon}(x_1),\eqlab{1dmappings2}\\
 x_1 &\mapsto \left(\widetilde{\mathcal X}_{L,\epsilon} \circ X_{L,\epsilon}^C\right)(x_1)\eqlab{1dmappings3}, 
\end{align}
%
each coincide with the $x_1$ components of 
\begin{align*}
 &x_1\mapsto \left(P_R^C\circ \Psi_R^{-1}\right)\left(\left(\epsilon \nu^{-1}\right)^{1/(2k+1)},x_1,\nu\right)\in \Sigma_L^{\text{out},C},\\
  &x_1\mapsto P_C\left(\left(\epsilon \nu^{-1}\right)^{1/(2k+1)},x_1,\nu\right)\in \Sigma_R^{\text{in},C},\\
 &x_1\mapsto \left(\Psi_L\circ P_L^C\right)\left(\delta^{1/2k},x_1,\epsilon \delta^{-(2k+1)/2k}\right)\in \Sigma_R^{\text{out}},
\end{align*}
respectively. Notice that the domain for \eqref{1dmappings3} -- following \eqref{sigmaLinC} -- is 
\begin{align}
 x_1 \in I_L^C:=[-\theta \chi^{-1} \epsilon^{2k/(2k+1)} ,\theta \chi^{-1} \epsilon^{2k/(2k+1)}].\eqlab{domainhere}
\end{align}
where
\begin{align}
\chi:= \delta \nu^{2k/(2k+1)},\eqlab{clR} 
\end{align}}

\qss{
The functions $X_{R,\epsilon}^C$, $\mathcal X_{R,\epsilon}$, $X_{C,\epsilon}$, $\widetilde{\mathcal X}_{L,\epsilon}$ and $\epsilon^{2k/(2k+1)}X_{L,\epsilon}^C$ (seeing that $X_{L,\epsilon}^C$ itself is singular as $\epsilon\rightarrow 0$) are each $C^2$ with respect to $x_1$, depending continuously on the small parameter $\epsilon\in [0,\epsilon_0)$. 
In particular,
the expressions in \eqref{X1C} and \eqref{XLC} give the following asymptotics for $\epsilon\rightarrow 0$:
\begin{align}
 X_{R,\epsilon}^C(x_1) &= \chi^{-1}\epsilon^{2k/(2k+1)} x_1+\mathcal O(\epsilon),\eqlab{XRepsAsymp}\\
 X_{L,\epsilon}^C(x_1)&=\chi \epsilon^{-2k/(2k+1)} x_1 + \mathcal O(\epsilon^{1/(2k+1)}),\eqlab{XLepsAsymp}
\end{align}
%
upon using $\rho_1=(\epsilon \nu^{-1})^{1/(2k+1)}$ on $\Sigma_R^{\text{in},C}$ and $\epsilon_1=\epsilon \delta^{-(2k+1)/2k}$ on $\Sigma_{L}^{\text{in},C}$, respectively.
On the other hand, by \eqref{XCrho}:
\begin{lemma}\lemmalab{XC1}
The following estimate holds
\begin{align*}
 X_{C,\epsilon}(x_1) = X_{C,0}(x_1) +\mathcal O(\epsilon^{1/(2k+1)}).
\end{align*}
\end{lemma}
In these expressions, the order of the remainder does not chance upon differentiation with respect to $x_1$. }


%

\qss{Analogously to the one-dimensional versions \eqref{1dmappings1}, \eqref{1dmappings2} and \eqref{1dmappings3} of $P_R^C\circ \Psi_R^{-1}$, $P_C$ and $\Psi_L\circ P_L^C$, respectively, the mapping $Q_1$ also becomes one-dimensional on the invariant sets defined by $\rho_1^{2k+1}\epsilon_1=\epsilon$:
\begin{align}
 x_1\mapsto X_\epsilon(x_1):=\left(X_{R,\epsilon}^C\circ \mathcal X_{R,\epsilon} \circ X_{C,\epsilon}\circ \widetilde{\mathcal X}_{L,\epsilon}\circ X_{L,\epsilon}^C\right)(x_1),\eqlab{QDefined} 
\end{align}
for all $x_1\in I_L^C$,
the right hand side of
\eqref{QDefined} being the $x_1$-entry of
\begin{align*}
 Q_1\left(\delta^{1/2k},x_1,\epsilon \delta^{-(2k+1)/2k}\right) = \begin{pmatrix}
                                                                   \delta^{1/2k}\\
                                                                   X_\epsilon(x_1)\\
                            \epsilon \delta^{-(2k+1)/2k}                                       
                                                                  \end{pmatrix},
\end{align*}}
 


\qss{
\begin{lemma}\lemmalab{Q111}
Consider any $c>0$, then there exist constants $\eta$, $\theta \in (0,\eta)$, $\delta_1$, and $\nu_1$ such that for any $\delta\in (0,\delta_1]$, $\nu\in (0,\nu_1]$ and all $x_1\in I_L^C$ the first and second derivative of the mapping \eqref{QDefined} satisfy
\begin{align}
 X_\epsilon'(x_1) &=\left(1+\mathcal O(c)\right)(X_{C,\epsilon})_{x_1}',\nonumber\\
 X_\epsilon''(x_1) &=\left(1+\mathcal O(c)\right) \chi \epsilon^{-2k/(2k+1)} ( X_{C,\epsilon})_{x_1x_1}'' +\mathcal O(1),\eqlab{Xc11}
\end{align}
respectively, for all $0\le \epsilon\ll 1$. 
Here $(X_{C})_{x_1}'$ and $(X_C)''_{x_1x_1}$ are each evaluated at the value of the right hand side in \eqref{1dmappings3}.
%
%
\end{lemma}
\begin{proof}
The result follows from a simple calculation using the chain rule on \eqref{QDefined} and the fact that $\widetilde{\mathcal X}_{L}$ and $\widetilde{\mathcal X}_{R}$, given in \eqref{tildex1x} and \eqref{x1PsiR}, are smooth and satisfy $\widetilde{\mathcal X}_{L}(0,0,0)=-1$, $\left(\widetilde{\mathcal X}_{L}\right)'_{x_1}(0,0,0)=1$ and $\widetilde{\mathcal X}_{R}(0,0,0)=1$, $\left(\widetilde{\mathcal X}_{R}\right)'_{x_1}(0,0,0)=1$. Taking $\eta$, $\delta_1$ and $\nu_1$ small enough, we can for any $\delta\in (0,\delta_1]$, $\nu\in (0,\nu_1]$ therefore ensure that ${\mathcal X}_{R,\epsilon}$ and $\widetilde{\mathcal X}_{L,\epsilon}$ are each $C^2$ $\mathcal O(c)$-close to their linearizations for all $0\le \epsilon\ll 1$. We take $\theta\in (0,\eta)$ small enough to ensure that $Q_1$ is well-defined. 
\end{proof}}

\begin{figure}
\begin{center}
\includegraphics[width=.85\textwidth]{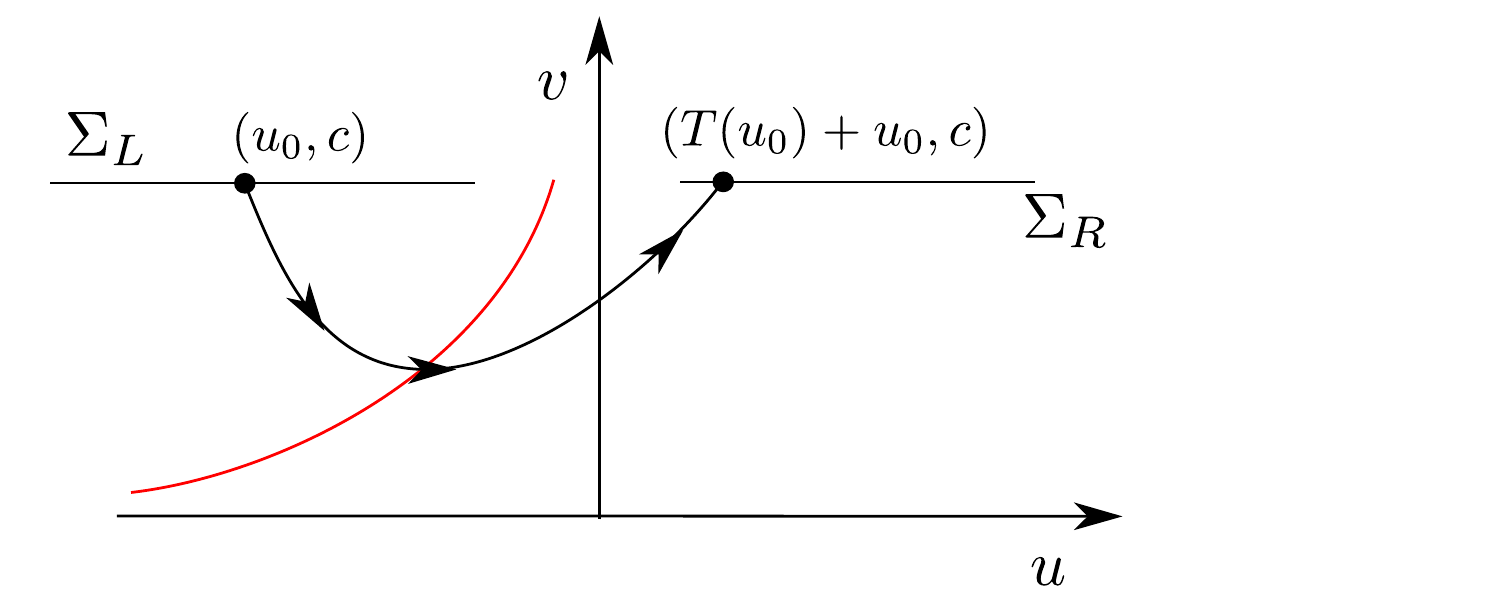} \end{center}
 \caption{Dynamics within the $(u,v)$-plane. We describe the transition map from $\Sigma_L$ and $\Sigma_R$ using a simple analysis of the variational equations, which in the $(u,v)$-variables, take a simple form.}
 \figlab{uvPlane}
              \end{figure}

\subsubsection*{\qss{Finishing the proof of \thmref{main1}(d)}}
\qss{
The result in \thmref{main1}(d) (ii) now follows upon combining \lemmaref{XC1} and \lemmaref{Q111} using \lemmaref{nu0}. Indeed, we have
\begin{lemma}\lemmalab{final1d}
 For any $c>0$ there exist constants $\eta$, $\theta\in (0,\eta)$, $\delta$ and $\nu$ such that the following holds for all $x_1\in I_L^C$ 
\begin{align*}
 X_{\epsilon}'(x_1)\in (1-c,c),\,X''_\epsilon(x_1) <-c^{-1},
\end{align*}
and all $0<\epsilon\ll 1$.
\end{lemma}
\begin{proof}
 Consider any $c>0$. Then our strategy for determining $\eta$, $\theta \in (0,\eta)$, $\delta$ and $\nu$ is as follows. First we fix $\eta$, $\theta\in (0,\eta)$ and $\delta=\delta_1$ as in \lemmaref{Q111}. Subsequently, we then follow \lemmaref{nu0} with $\zeta=\theta$ and pick $\nu=\text{min}(\nu_0,\nu_1)$. Then -- upon taking $0<\epsilon\ll 1$ -- it follows by \lemmaref{XC1} and \lemmaref{Q111}, and using the large $\epsilon^{-2k/(2k+1)}$ prefactor in \eqref{Xc11}, that
 \begin{align}
  X'_{\epsilon}(x_1)\in (-1+\mathcal O(c),\mathcal O(c)),\quad
  X''_{\epsilon}(x_1)<-c^{-1},\eqlab{intervalhe}
 \end{align}
 for all $x_1\in I_L^C$. 
\end{proof}}

To prove (i) in \thmref{main1}(d), we sketch the argument as follows. First, we work in chart $(\bar \rho=1)_1$ near $p_L$ using the coordinates in \lemmaref{rho1x1eps1Tilde}. We then describe a mapping from $\Sigma_L^{\text{in},L}$, consisting of all those points in $\Sigma_L^{\text{in}}$ with $x_1<-\theta(\epsilon_1 \nu^{-1})^{2k/(2k+1)}$, recall \eqref{sigmaLinC}, to the section 
\begin{align*}
 \Sigma_L^{\text{out},L} = \{(\rho_1,x_1,\epsilon_1)\vert x_1 = -\nu,\,\rho_1 \in[0,\beta_1],\,\epsilon_1\in [0,\beta_2]\}.
\end{align*}
This gives an expanding map, but as before, this expansion is (more than) compensated by the contraction eventually gained at $p_R$. Essentially, the result therefore follows from the details of the mapping from $\Sigma_L^{\text{out},L}$ to $\Sigma_R^{\text{in},C}$. This mapping can be described in two parts. The first part consists of a simple, near-identity mapping, near $(\bar \rho,\bar x,\bar \epsilon)=(0,-1,0)$, which can be studied in the chart $\bar x=-1$, the details of which are standard and left out of this manuscript completely for simplicity. The second part, is described in $(\bar \epsilon=1)_2$ using the local dynamics near the nonhyperbolic saddle $p_a$. This mapping is contracting due to the exponential contraction towards the center manifold that extends the invariant manifold $S_\epsilon$ onto the blowup sphere, recall \lemmaref{M1} item (1). In combination, this then proves (i). 

(iii) in \thmref{main1}(d) is simpler and can be described in the chart $(\bar r=1)_1$ only. For this, we again apply $\Phi_i^{-1}$ near $p_i$, $i=L,R$, respectively, and define a mapping from $\Sigma_L^{\text{in},R}$, consisting of all those points in $\Sigma_L^{\text{in}}$ with $x_1>\theta (\epsilon_1\nu^{-1})^{2k/(2k+1)}$, to 
\begin{align*}
 \Sigma_L^{\text{out},R} = \{(\rho_1,x_1,\epsilon_1)\vert x_1 = \eta,\,\rho_1 \in[0,\beta_1],\,\epsilon_1\in [0,\beta_2]\}.
\end{align*}
We subsequently follow this by a regular finite time flow map from $\Sigma_L^{\text{out},R}$ to 
\begin{align*}
 \Sigma_R^{\text{in},L} = \{(\rho_1,x_1,\epsilon_1)\vert x_1 = -\eta,\,\rho_1 \in[0,\beta_1],\,\epsilon_1\in [0,\beta_2]\},
\end{align*}
and lastly by a local mapping from $\Sigma_R^{\text{in},L}$ to $\Sigma_R^{\text{out}}$, recall \figref{barRhoEq1}. In this case, the mapping from $\Sigma_L^{\text{out},R}$ to $\Sigma_R^{\text{in},L}$ -- as in the proof of \lemmaref{nu0} -- can be made as close to the identity in $\rho_1$ and $\epsilon_1$ as desired. In combination, this proves (iii) and the proof \thmref{main1}. In particular, we highlight that the domains of (i) and (ii) as well as of (ii) and (iii) can (and will) be chosen to overlap.


\begin{proof}[\qss{Proof of \lemmaref{PC}}]
Following \eqref{PCDefine}, we prove \lemmaref{PC} by estimating $T'(u)$ and $T''(u)$. We will do so consecutively in the following.

Let $c = (\phi_+ \nu^k)^{-2/(2k+1)}$ and write the solution of \eqref{chini} with initial conditions $(u_0,c)\in \Sigma_L$ as $(u(t,u_0),v(t,u_0))$, i.e. $(u(0,u_0),v(0,u_0))=(u_0,c)$ for any $u_0$. Then $T(u_0)>0$ is the least positive solution satisfying
\begin{align}
 v(T(u_0),u_0) = c.\eqlab{Tdefine}
\end{align}
Differentiating \eqref{Tdefine} gives
\begin{align}
 T'(u_0) = -\frac{v'_u(T(u_0),u_0)}{v'_t(T(u_0),u_0)}.\eqlab{TpEqn}
\end{align}
Notice that if
\begin{align}
v'_u(T(u_0),u_0)<2v'_t(T(u_0),u_0),\eqlab{ineqFinal}
\end{align}
then,
by \eqref{TpEqn},
\begin{align}
 T'(u_0)>-2.\eqlab{Tpu0}
\end{align}
Since $P'_C(u)<0$ is trivial, this inequality implies the first claim in \eqref{UC} by differentiating \eqref{PCDefine}.  To show \eqref{ineqFinal} let $v_1(t)=v'_u(t,u_0)$ so that $v_1(0)=0$ and
\begin{align}
 \dot v_1 = 2 - k v^{-k-1}   v_1.\eqlab{v1Eqn}
\end{align}
Clearly, $v_1(t)>0$ for all $t>0$. 
Now, $\dot v(t,u_0) = v'_t(t,u_0)$ also satisfies the equation \eqref{v1Eqn}, but with the initial condition $\dot v(0,u_0)  = 2u_0+c^{-k}<0$. In light of \eqref{ineqFinal}, we therefore write $v_1$ as
\begin{align}
 v_1 = 2\dot v-w.\eqlab{w1Definition}
\end{align}
A simple calculations, shows that $w=w(t,u_0)$ also satisfies \eqref{v1Eqn}:
\begin{align}
 \dot w &=2-kv^{-k-1}w,\eqlab{wEqn}
\end{align}
 but now $w(0,u_0)=2\dot v(0,u_0)<0$. 
  We will now show that for all $u_0$
 \begin{align}
  0<w(T(u_0),u_0).\eqlab{ineqFinal2}
 \end{align}
  Adding $v_1=2\dot v (T(u_0),u_0)-w(T(u_0),u_0)$ to both sides of this equation, using \eqref{w1Definition}, produces \eqref{ineqFinal} and therefore \eqref{Tpu0} as desired.

 For \eqref{ineqFinal2}, notice by \eqref{wEqn} that $\dot w\ge 2$ for all $w\le 0$. Hence there exists a unique $t_*(u_0)$ such that
 \begin{align}
  w(t_*(u_0),u_0) = 0.\eqlab{wTstar}
 \end{align}
 Notice that $t_*(u_0)$ is a smooth function of $u_0$ by the implicit function theorem. Next, for values of $u_0<-c^{-k}/2$, sufficiently close to $-c^{-k}/2$ (the value of $u$ when the $v$-nullcline intersects $v=c$), a simple calculation shows that 
 \begin{align*}
  T'(u_0) = -2 -\frac23 k c^{-k-1} \left({c^{-k}} +2u_0\right) + \mathcal O\left(\left(c^{-k}/2 +u_0\right)^2\right) >-2,
 \end{align*}
 and $t_*(u_0)<T(u_0)$.
Hence, for these values of $u_0$,  it follows that \eqref{ineqFinal2}, and therefore also \eqref{Tpu0}, holds. 
 
 Suppose that 
 upon decreasing $u_0$ we find a first $u_*$ such that $T(u_*)=t_*(u_*)$. Notice then by \eqref{TpEqn} and \eqref{w1Definition} that $T'(u_*)=-2,$ and 
 \begin{align}
 t_*'(u_*)\le -2,\eqlab{contradict}
 \end{align}
 since $t_*(u_0)<T(u_0)$ for all $u_0\in (u_*,-c^{-k})$, by assumption. Then, by differentiating \eqref{wTstar},
 \begin{align}
  t_*'(u_*) = -\frac{w_1(t_*(u_*),u_*)}{\dot w(t_*(u_*),u_*)} = -\frac{w_1(t_*(u_*),u_*))}{2}.\eqlab{tstarp} \end{align}
 where $w_1$ satisfies the equation
\begin{align}
 \dot w_1 &=-kv^{-k-1} w_1 +k(k+1) v^{-k-2} w v_1, \eqlab{z1Eqn}
\end{align}
with $w_1(0) = 4$. In \eqref{tstarp}, we have also used that $\dot w=2$ at $t=t_*$ where $w=0$. 
Since $w(t,u_*)<0$ for all $t\in [0,t_*(u_*))$, and $v_1(t,u_*)>0$ for all $t$, we have by \eqref{z1Eqn} that $w_1(t_*(u_*),u_*)<4$. But then by \eqref{tstarp}
\begin{align*}
 t_*'(u_*)>-2.
\end{align*}
However, this contradicts \eqref{contradict} and hence no $u_*$ exists. Consequently, $T(u_0)>t_*(u_0)$ for all $u_0<-c^{-k}/2$ and therefore \eqref{Tpu0} holds.


%
%


For the subsequent claim in \eqref{UC}, we obtain the following expression for $T''(u_0)$
\begin{align}
 T''(u_0) = -\frac{1}{v'_t(T(u_0),u_0)}\left(2P'_C(u_0)(P'_C(u_0)-1)+v_2(T(u_0))\right),\eqlab{T2Expr}
\end{align}
by differentiating \eqref{Tdefine} twice with respect to $u_0$,
where $v_2 (t)= v''_{uu}(t,u_0)$ satisfies $v_2(0)=0$ and 
\begin{align*}
 \dot v_2 =- k v^{-k-1}  v_2+ k(k+1)v^{-k-2}  v_1^2 .
\end{align*}
Since $\dot v_1(0)=2$ and $v_1(t)>0$, it follows that $v_2(t)>0$ for all $t>0$. Consequently, by the first property in \eqref{UC}, \eqref{T2Expr} gives $T''(u_0)<0$, proving the last property in \eqref{UC}.  

\end{proof}
\section{Proof of \thmref{main2}}\seclab{proof2}
\qs{We now move on to prove \thmref{main2}. 
The periodic orbits we describe are fix points of $P(\cdot,\alpha,\epsilon)$, recall \eqref{QR}. 
\begin{lemma}\lemmalab{finallem}
We have:
\begin{itemize}
 \item[(1)] Fix $\alpha<0$ sufficiently small. Then locally there exists two and only two fix points of $P(\cdot,\alpha,\epsilon)$ for all $0<\epsilon\ll 1$. 
 \item[(2)] Fix $\alpha>0$ sufficiently small. Then no local fix points exist of $P(\cdot,\alpha,\epsilon)$ for all $0<\epsilon\ll 1$.
\end{itemize}
\end{lemma}}
\begin{proof}
\qs{We first prove (1). Consider therefore $\alpha<0$ sufficiently small. Then we obtain, by the properties in \lemmaref{Rmap}, using \lemmaref{regular}, that $L\ni x\mapsto P(x,\alpha,\epsilon)$ is a strong contraction within some open subset $L\subset K$, recall \thmref{main1}(c). This produces a hyperbolic and attracting fix-point of $P$ in $L$ by the contraction mapping theorem. This fix-point for $\alpha<0$ co-exists with the hyperbolic and repelling fix-point obtained by the regular perturbation of $\Gamma_\alpha\subset \{y\ge c(\alpha)>0\}$ of $Z_+$, for some $c(\alpha)>0$, using \lemmaref{regular}. The fact that only two intersections exist is a simple consequence of the properties of $Q$ in \thmref{main1}(d). Case (2) in \lemmaref{finallem} is also a simple corollary of \lemmaref{Rmap} and \thmref{main1}(d).}
\end{proof}
\qs{\begin{remark}\remlab{bonet}
    \cite[Theorem 2.3]{bonet-rev2016a} is basically \lemmaref{finallem} for the Sotomayor-Teixera regularization functions considered in that paper. It rests, as the proof \lemmaref{finallem}, upon the domains (i) and (iii) in \thmref{main1} (d). But as we shall see below, the intermediate domain (ii) in \thmref{main1} (d) allow for a description of what occurs in between the cases (1) and (2) in \lemmaref{finallem}. In fact, the results of \thmref{main1} enable a very direct approach for this regime using the implicit function theorem.  
    \end{remark}}

Following \eqref{QRFixPoint}, fix points are intersections of the two graphs of $Q(\cdot,\alpha,\epsilon)$ and $R^{-1}(\cdot,\alpha,\epsilon)$. In particular, transverse intersections correspond to hyperbolic periodic orbits.  We illustrate this viewpoint in \figref{QRgraph} in the case where no intersections occur, e.g. for $\alpha>0$ and $0<\epsilon\ll 1$, recall \lemmaref{finallem} (2). Following \thmref{main1}(d) and \lemmaref{Rmap}, it is easy to use this graphical framework together with \eqref{concaveDown} to prove the existence of a unique saddle-node bifurcation, occurring within the domain (ii) of \thmref{main1}(d). Notice in particular, by \lemmaref{Rmap}, see also \eqref{Rprop2}, that the slope of $R^{-1}(\cdot,\alpha,\epsilon)$, being the right hand side of \eqref{QRFixPoint}, is greater than $-1+\omega$, say, for some (new) fixed $\omega>0$, and all $\epsilon\in [0,\epsilon_0)$, $\alpha \in I$. From the picture in \figref{QRgraph}, the saddle-node is by \lemmaref{Rmap}, see \eqref{Rprop3}, obtained by decreasing $\alpha$. However, to obtain \eqref{alphaEpsSN} we proceed using an implicit function theorem argument as follows: \qss{Consider $Q_1$ \eqref{Q1comp} and the associated one-dimensional map \eqref{QDefined} obtained on the level sets defined by $\rho_1^{2k+1}\epsilon_1=\epsilon$. However, in the context of \thmref{main2}, we now write the right hand side of \eqref{QDefined} as $X_\epsilon(x_1,\alpha)$ highlighting the dependency on the bifurcation parameter $\alpha$. Then introduce $x_2$ and $\alpha_2$ by }
\begin{align}
x_1=\epsilon^{2k/(2k+1)}x_2, \quad \alpha=\epsilon^{2k/(2k+1)}\alpha_2,\eqlab{x2alpha2}
\end{align}
and let
\begin{align}
 Q_2(x_2,\alpha_2,\epsilon) :=\epsilon^{-2k/(2k+1)} X_\epsilon(\epsilon^{2k/(2k+1)}x_2,\epsilon^{2k/(2k+1)}\alpha_2,\epsilon).\eqlab{Q2}
\end{align}

\begin{lemma}\lemmalab{Q2expr}
 \qss{$Q_2$ satisfies the following estimates
 \begin{align}
  Q_2(x_2,\alpha_2,\epsilon) = \chi^{-1}\left(\mathcal X_{R,0} \circ X_{C,0} \circ  \widetilde{\mathcal  X}_{L,0} \right)(\chi x_2)+\mathcal O(\epsilon^{1/(2k+1)}),\eqlab{Q2Asymp}
 \end{align}
 where $x_2 \in [-\theta \chi^{-1},\theta \chi^{-1}]$, recall \eqref{domainhere}.
%
 The order of the remainder in \eqref{Q2Asymp} does not change as $\epsilon\rightarrow 0$ upon differentiation with respect to $x_2$ and $\alpha_2$. In particular, for any $c>0$ there exists $\eta$, $\theta\in (0,\eta)$,$\delta$ and $\nu$ such that
 \begin{align*}
 (Q_2)'_{x_2}(x_2,\alpha_2,0)& = \left(\mathcal X_{R,0} \circ X_{C,0} \circ  \widetilde{\mathcal  X}_{L,0} \right)'(\chi x_2)\in (-1+c,-c),\\
 (Q_2)''_{x_2x_2}(x_2,\alpha_2,0)& = \chi\left(\mathcal X_{R,0} \circ X_{C,0} \circ  \widetilde{\mathcal  X}_{L,0} \right)''(\chi x_2)<0
 \end{align*}}
   \end{lemma}
 \begin{proof}
  Follows from the definition of $Q_2$ and \lemmaref{final1d}, recall also \eqref{XRepsAsymp} and \eqref{XLepsAsymp}. 
 \end{proof}


%

 Next, we continue to write the smooth $R^{-1}(\cdot,\alpha,\epsilon)$ in the $(\bar r=1)_1$-chart and express the resulting mapping in the same coordinates used for $Q_1$ in \eqref{Q1comp} by composing the resulting mapping with $\Psi_L$ from the left and $\Psi_R^{-1}$ from the right. Subsequently, define $R_2^{-1}(x_2,\alpha_2,\epsilon)$ completely analogously to \eqref{Q2} by first going to a one-dimensional mapping on the level sets $\rho_1^{2k+1}\epsilon_1=\epsilon$, insert \eqref{x2alpha2} into the resulting expression and ``blow up'' $\gamma_R$ by dividing the resulting expression by $\epsilon^{-2k/(2k+1)}$ as in \eqref{Q2}. 
 \begin{lemma}\lemmalab{R2expr}
 $R_2^{-1}$ is smooth in $x_2$ and $\alpha_2$. Also $R_2^{-1}$ satisfy the following estimate
 \begin{align}
  R_2^{-1}(x_2,\alpha_2,\epsilon) = \upsilon_0 x_2+\upsilon_1\alpha_2 +\mathcal O(\epsilon^{1/(2k+1)}),\eqlab{R2expr}
 \end{align}
 where $\upsilon_0\in (-1,0)$, $\upsilon_1 >0$ upon possibly decreasing $\eta$, $\nu$ and $\delta$. 
 \end{lemma}
\begin{proof}
 Simple consequence of the smoothness of $R$, the fact that $R(\gamma_R,0,0)=\gamma_L$ and finally that $\mathcal X_R(\rho_1,\cdot,\epsilon_1)$ and $\widetilde{\mathcal X}_L(\rho_1,\cdot,\epsilon_1)$ are as close to their linearization as desired upon decreasing $\xi$. In particular, the estimates for $\upsilon_0$ and $\upsilon_1$ then follow from the fact
that $R'_x(\gamma_R,0,0)^{-1}\in (-1,0)$, $R'_\alpha(\gamma_R,0,0)^{-1}>0$ by \lemmaref{Rmap}.
\end{proof}

 The fix-point equation \eqref{QRFixPoint} becomes
\begin{align}
 Q_2(x_2,\epsilon,\alpha) = R_2^{-1}(x_2,\alpha_2,\epsilon),\eqlab{Q2R2}
\end{align}
in terms of $Q_2$ and $R_2^{-1}$. This equation is suitable for the application of the implicit function theorem. In particular, each side is at least $C^2$ with respect to $x_2$ and $\alpha_2$, depending continuously on $\epsilon\in [0,\epsilon_0)$.
\begin{lemma}\lemmalab{sn2}
For each $\epsilon \in [0,\epsilon_0)$, there exists a unique solution $(x_2(\epsilon),\alpha_2(\epsilon),\epsilon)$ of \eqref{Q2R2}, $x_2(\epsilon)$ and $\alpha_2(\epsilon)$ both being continuous in $\epsilon \in [0,\epsilon_0)$, such that 
 \begin{align}
 (Q_{2})'_{x_2} &= (R_2^{-1})'_{x_2},\quad   \text{(Degeneracy condition)},\eqlab{Q2R2p}\\
   (Q_{2})''_{x_2x_2} &\ne (R_2^{-1})''_{x_2x_2},\quad \text{(Nondegeneracy condition I)},\eqlab{Q2R2c1}\\
    (Q_{2})'_{\alpha_2} &\ne (R_2^{-1})'_{\alpha_2},\quad \text{(Nondegeneracy condition II)},\eqlab{Q2R2c2}
  \end{align}
  with all partial derivatives evaluated at $(x_2(\epsilon),\alpha_2(\epsilon),\epsilon)$.
\end{lemma}
\begin{proof}
Consider $\epsilon=0$. Then by \lemmaref{Q2expr} and \lemmaref{R2expr}, \eqref{Q2R2} becomes
\begin{align}
Q_2(x_2,\alpha_2,0) - \upsilon_0 x_2-\upsilon_1\alpha_2=0,\eqlab{sneq1}
\end{align}
Furthemore, \eqref{Q2R2p} becomes
\begin{align}
(Q_2)'_{x_2}(x_2,\alpha_2,0) - \upsilon_0=0.\eqlab{sneq2}
\end{align}
Seeing that $\upsilon_0\in (-1,0)$, it follows by \lemmaref{Q2expr} -- taking $c>0$ so small that $\upsilon_0 \in  (-1+c,c)$ -- that that the equation \eqref{sneq2} has a unique solution $x_2\in [-\theta \chi^{-1},\theta \chi^{-1}]$. Inserting this into \eqref{sneq1} gives a unique $\alpha_2$ since $\upsilon_1\ne 0$.
Notice that the resulting $\epsilon=0$ solution  $(x_2,\alpha_2)$ satisfies the nondegenaracy conditions \eqref{Q2R2c1} and \eqref{Q2R2c2}. In particular, by computing the Jacobian of the left hand sides of \eqref{sneq1} and \eqref{sneq2} with respect to $(x_2,\alpha_2)$, we obtain the matrix
\begin{align*}
\begin{pmatrix}
 (Q_2)'_{x_2}(x_2,\alpha_2,0) & -\upsilon_1\\
(Q_2)''_{x_2x_2}(x_2,\alpha_2,0) & 0
\end{pmatrix}
\end{align*}
which is regular by \lemmaref{Q2expr}. The existence of a unique continuous solution $(x_2(\epsilon),\alpha_2(\epsilon))$ for all $0<\epsilon\le \epsilon_0$ with the described properties therefore follows from the implicit function theorem. 
\end{proof}
Returning to $x$ and $\alpha$, the solution in \lemmaref{sn2} becomes $x=\epsilon^{2k/(2k+1)}x_2(\epsilon),\,\alpha= \epsilon^{2k/(2k+1)}\alpha_2(\epsilon)$ as desired. This is the unique saddle-node bifurcation of \eqref{QR}, since \eqref{Q2R2p}, \eqref{Q2R2c1} and \eqref{Q2R2c2} imply \eqref{QRcond0}, \eqref{QRcond1} and \eqref{QRcond2}, respectively. The proof of \thmref{main2} is therefore complete.
\begin{figure}
\begin{center}
\includegraphics[width=.7\textwidth]{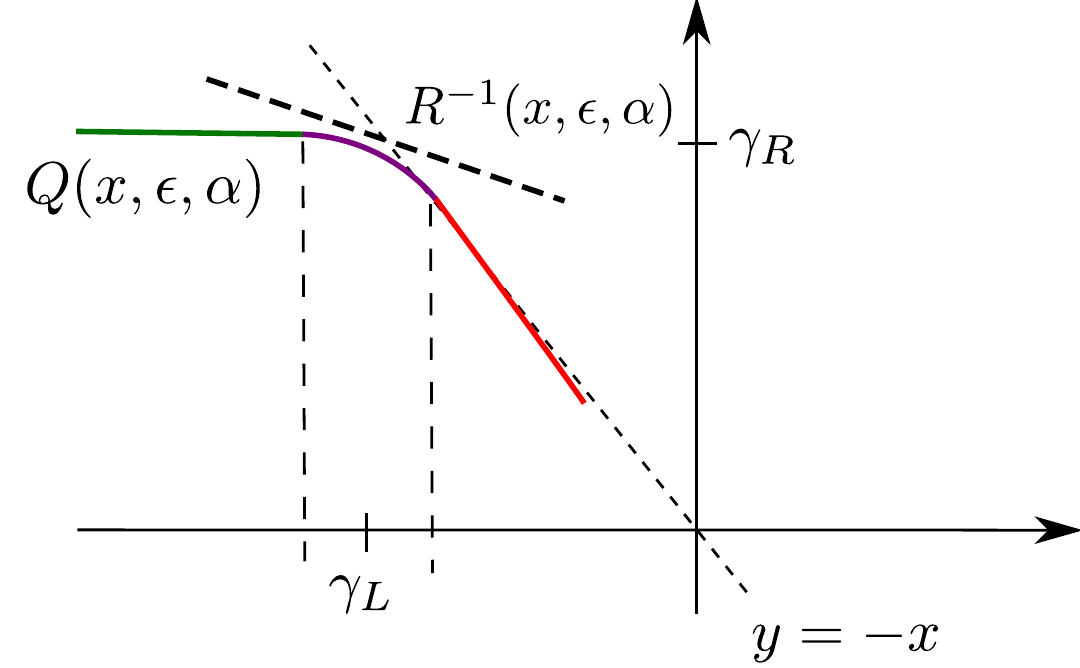} \end{center}
 \caption{Graphs of $Q$ (thick curve in different colours) and $R^{-1}$ (thick dashed line), following \thmref{main1}(d) and \lemmaref{Rmap}. The different colours on the graph of $Q$, represents the different domains in \thmref{main1}(d): Green for the domain in  (i), purple for the domain in (ii), which is $\mathcal O(\epsilon^{2k/(2k+1})$-close to $\gamma_L$, red for the domain in (iii). By \lemmaref{Rmap}, a unique quadratic tangency between these graphs occur within the purple domain, upon adjusting $\alpha$. This tangency is the saddle-node bifurcation of limit cycles. }
 \figlab{QRgraph}
              \end{figure}

              \section{Discussion}\seclab{disc}
             
             \subsection*{Comparison with previous results}
             \qss{\cite{bonet-rev2016a} also describes the regularization of the visible fold, but only for the class of regularizations in \eqref{psi}. They also assume an algebraic condition like (A2) but at $s=\pm 1$ rather than asymptotically. \cite[Theorem 2.2]{bonet-rev2016a} describes the invariant manifold's intersection with a fixed section, similar to \thmref{main1}(b). For their regularization functions, $m(\epsilon)=\gamma_R+\mathcal O(\epsilon)$ to leading order for any $k$. (In \cite[Thoerem 3.3]{kristiansen2017a} this result was rediscovered using blowup. This also allowed for a more detailed expression of the remainder; it is a smooth function of $\epsilon^{1/(2k+1)}$.) \qs{Notice that in our case, $m(\epsilon)=\gamma_R+\mathcal O(\epsilon^{2k/(2k+1})$, see \thmref{main1}(b), where $k=k_+$ is the decay rate in (A2). Going through the details, we may also realise that it is a smooth function of $\epsilon^{1/(2k+1)}$ but in general only $C^l$, $l=l(k)\ge 1$, due to the resonances (producing $\log$'s), recall \eqref{strongres}. }}
             
             \cite[Theorem 2.4]{bonet-rev2016a} also addresses the grazing bifurcation, described in the present manuscript in \thmref{main2}, but neither \cite{bonet-rev2016a}, \qss{nor the conference paper \cite{sutmig} by the same authors}, did rigorously establish existence nor uniqueness of a saddle-node bifurcation.       
             Our results improve on their findings in several ways. Firstly, we provide more details close to the tangency, see \thmref{main1}(d), for a more general, and more ``practical'', class of regularization functions. In turn, this allowed us to prove the existence and uniqueness of a saddle-node bifurcation in the regularized grazing bifurcation. Interestingly, in \cite{bonet-rev2016a}, the equation $y'(x) = x+y^k$, where $k$ is the order of the tangency at $s=1$, recall \eqref{psi}, plays an important role. For our regularization functions, satisfying (A1)-(A2), this equation is replaced by a Chini equation $v'(u)=2u+v^{-k}$, see \eqref{chini}, where $k=k_+$ is the decay rate in (A2). 
             
             \subsection*{The blowup approach}             \qss{Our approach to the problem was to use blowup \cite{dumortier_1996,krupa_extending_2001} following \cite{kristiansen2018a}. 
              This framework is general, see also the outlook below,  and has all the benefits that come with using blowup. The blowup method frequently provides remarkably detailed insight into complicated global dynamical phenomena, often even in a constructive way, see \cite{bossolini2017a,Gucwa2009783,kosiuk2011a, kosiuk2015a,2019arXiv190312232U,kristiansen2019a}. For the planar visible fold studied in the present paper, we applied two consecutive blowups, blowing up the discontinuity set to a cylinder and the visible fold to a sphere, and obtained improved hyperbolicity properties. This allowed us to obtain the desired results by combining:
              \begin{itemize}
              \item[(a)] The existence of an invariant manifold, see \lemmaref{M1} item (1); 
              \item[(b)] The (special) linearization used in \secref{r1} near the hyperbolic points $p_{L}$ and $p_R$, recall \lemmaref{rho1x1eps1Tilde};
              \end{itemize}
              -- both based on classical local results in dynamical systems theory (center manifold theory, linearization, etc.) -- with a global analysis of the Chini equation \eqref{chini}. In contrast to e.g. \cite{kristiansen2018a}, we use blowup to study -- for the first time -- a very general class of regularization functions, including functions such as the Goldbeter-Koshland function \eqref{gkfunc} that appear in applications.}

            \qs{We highlight that the Sotomayor-Teixera regularization functions used in \cite{bonet-rev2016a} are easier to study in general than the ones satisfying (A1)-(A2). In fact, a proper cylindrical blowup of $y=\epsilon=0$ is strictly speaking not necessary for these systems. Indeed, the PWS system coincides with the regularized one outside $\vert y\vert \le \epsilon$ for these functions and one does therefore not lose compactness in the ``scaling chart'' $(\bar \epsilon=1)_2$ \eqref{r2y2} where $y=\epsilon y_2$. In fact, one can just restrict to $y_2\in [-1,1]$. Consequently, for these systems, the directional charts $(\bar y=1)_1$ and $(\bar y=-1)_3$, recall \eqref{r1epsilon1} and \eqref{r3epsilon3}, are not necessary for the global picture. This is exploited in \cite{bonet-rev2016a,krihog,krihog2}. As a consequence, the blowup of $\overline T$ for the Sotomayor-Teixera functions have a ``scaling chart'', where the variables local to $(x,y_2,\epsilon)=(0,1,0)$ are just scaled by appropriate roots of $\epsilon$. In our case, the blowup of $\overline T$ is more complicated, as none of the directional charts, see \eqref{epsilon11} and \eqref{epsilon12}, correspond to pure scalings. I therefore believe that it is not clear how one would reproduce the results of our paper (rigorously) using asymptotic methods.}
             
             \subsection*{Regularization functions}

In this paper, we have considered a wide class of regularization functions. However, it may seem restrictive that \eqref{phiLimit} holds for all $\epsilon$. But if 
\begin{align}
 \phi(y_2,\epsilon) \rightarrow \left\{\begin{array}{cc} b(\epsilon) & \text{for $y_2\rightarrow \infty$},\\
                                 a(\epsilon) &  \text{for $y_2\rightarrow -\infty$},
                                \end{array}\right.\eqlab{phiLimit2}
\end{align}
for example, with $a(0)<b(0)$, then 
\begin{align*}
 \tilde \phi(y_2,\epsilon) = \frac{\phi(y_2,\epsilon)-a(\epsilon)}{b(\epsilon)-a(\epsilon)}.
\end{align*}
satisfies \eqref{phiLimit}. We can then write $Z(z,\phi,\alpha) = \tilde Z(z,\tilde \phi,\epsilon,\alpha)$ and subsequently $\tilde Z$ in the form \eqref{fA}. Notice, that following this procedure $\tilde Z_\pm$ will in general depend upon $\epsilon$. However, only to keep the notation as simple as possible, we focused on the simpler setting in \eqref{fA}.

The class of functions described by (A1)-(A2) include standard regularization functions, such as $\arctan$ and \eqref{phiexample} but also nontrivial ones like the Goldbeter-Koshland function \eqref{gkfunc}
where $k_\pm=1$. Functions like $\tanh$ and \eqref{phicor}, where $k_\pm=\infty$ in (A2), are more difficult, because the blowup method does not apply directly, but they can, at least when the remainder is exponential, be tackled using the technique in \cite{kristiansen2017a}, see e.g. \cite[Theorem 3.5]{kristiansen2017a}.
%

Qualitatively, different regularization functions within our class do not change the results. However, the decay rate $k=k_+$ in (A2) does change the details quantitatively, see \eqref{mEpsSM} and \eqref{alphaEpsSN}.




If the regularization function $\phi$ is not monotone, such that (A1) and \eqref{monone} are not satisfied, then the critical manifold $\overline S$ upon blowup, will have folds, where, working in the scaling chart $(\bar \epsilon=1)_2$, see \eqref{scalingEqns}, classical results from singular perturbation theory can be applied, e.g. \cite{krupa_extending_2001} or canard theory \cite{krupa_relaxation_2001}. ($\phi(s,\epsilon) = s/\sqrt{s^2+\epsilon s+1}$ is an example with a fold $s=-2\epsilon^{-1}\rightarrow - \infty$ for $\epsilon\rightarrow 0$. In such cases, additional blowups in the directional charts are probably required to resolve such phenomena.)  \qs{In particular, due to such folds, the critical manifold will be repelling close to the point $T$. 
See also \cite{bossolini2017b,kristiansen2019d} for applications; in these cases Filippov does not agree with the $\epsilon\rightarrow 0$ limit. They can, however, be studied by the same methods used in the present paper. The generalization of \thmref{main2} will then include a canard phenomenon where the unstable limit cycles can be extended beyond the tangency point until they reach the fold of the critical manifold. We leave the details of this case to future work.}
Similarly, if (A0) does not hold, then the scaling chart, obtained by setting $y=\epsilon y_2$, is simply a general slow-fast system; it can be as complicated as any planar slow-fast system and there is little value in making the PWS connection in a general framework. 

\subsection*{Outlook}

             \qs{Going forward, I have together with co-authors applied the approach in the present paper to two systems in applications, see  \cite{kristiansen2019d} and \cite{kukszm}. In particular, in \cite{kristiansen2019d} we study singularly perturbed oscillators with exponential nonlinearities of the form $e^{y\epsilon^{-1}}$. These systems have nonsmooth limits upon normalization and can be written in the form \eqref{ztf} with $\phi$ as in \eqref{phicor}. Due to the exponentials in \eqref{phicor}, however, we apply a slight modification of the approach used in the present manuscript (by following  \cite{kristiansen2017a}) to handle the resulting $k_\pm =\infty$ in (A2). In \cite{kukszm}, on the other hand, we study substrate-depletion oscillators which have a switch-like function modelling an autocatalytic process. In \cite{Tyson2003}, this switch is described by the Goldbeter-Koshland function \eqref{gkfunc}. We combine the method used in the present paper with a parameter space blowup to describe the existence and nonexistence of limit cycles in the substrate-depletion oscillators close to their PWS limits. We find that the main mechanism for the oscillations is due to a ``boundary node bifurcation'' \cite{Kuznetsov2003, hogan2016a} where a node of the PWS system intersects the discontinuity set. Our spherical blowup of a fold point, like $\overline T$ in the present paper, captures this transition.} 
             
             \qs{In \cite{jelbart}, we consider a related situation of a ``boundary focus bifurcation'' \cite{Kuznetsov2003} in the friction oscillator problem for $\alpha=0$, recall \secref{model}. For $0<\epsilon\ll 1$ this PWS bifurcation gives rise to an additional Hopf bifurcation (now supercritical) on a blowup sphere (like the blowup of $\overline T$ in the present paper). We again combine the blowup approach promoted in this paper with a parameter space blowup with the overall aim to connect (in a smooth family) the attracting Hopf cycles all the way up to the $\mathcal O(1)$ cycles that bifurcate near the PWS grazing bifurcation, as described in \corref{model}. 
                                      }

\section*{Acknowledgement}
I would like to thank Samuel Jelbart for bringing the friction oscillator and the model \eqref{muModel} to my attention, for sharing references on the subject and for providing valuable feedback on an earlier version of the manuscript.  
%
\appendix 
\section{Proof of \thmref{main1}(a)}\applab{appA}
The analysis in the chart $\bar y=-1$, see \eqref{r1epsilon1}, plays little role and is similar to how we deal with $\bar y=1$. The details are therefore omitted. In the following we therefore consider $\bar \epsilon=1$ and $\bar y=1$ only. 
\subsection{Chart $\bar \epsilon=1$}
In this chart, we obtain the following equations
\begin{align}
 \dot x &=\epsilon \phi(y_2,\epsilon)(1+f(x,\epsilon y_2)),\eqlab{scalingEqns}\\
 \dot y_2 &= \phi(y_2,\epsilon)(2x+\epsilon y_2 g(x,\epsilon y_2)) + 1-\phi(y_2,\epsilon),\nonumber
\end{align}
and $\dot r_2=0$, 
using \eqref{ZpmNF} and \eqref{ztf} in \eqref{zext}. This is a slow-fast system with $x$ slow and $y_2$ fast. Setting $\epsilon=0$ gives the layer problem where $x$ is a parameter and 
\begin{align*}
 \dot y_2 &=\phi(y_2,0)2x+1-\phi(y_2,0),
\end{align*}
and hence a normally hyperbolic and attracting, but noncompact, critical manifold:
\begin{align}
 S =\{(x,y_2)\in U_\xi \vert \phi(y_2,0) =\frac{1}{1-2x},x<0\}.\eqlab{Schart2}
\end{align}
Fixing $J=[-\xi,-\nu]$, we can apply Fenichel's theory to $S\cap \{x\in J\}$ and conclude the existence of the invariant manifold $S_\epsilon$ in \thmref{main1}(a). The fact that $Z\vert_{S_\epsilon}$ is a regular perturbation of the Filippov vector-field is standard and follows easily from the fact that the reduced problem on $S$: 
\begin{align*}
 x' &=\phi(y_2,0)(1+f(x,0)) = \frac{1}{1-2x}(1+f(x,0)),
\end{align*}
coincides with Filippov. 
\subsection{Chart $\bar y=1$}
Upon inserting \eqref{r1epsilon1} into \eqref{zext}, using \eqref{ztf} and \eqref{ZpmNF}, we obtain the following equations
\begin{align*}
\dot r_1 &=-r_1 F(r_1,x,\epsilon_1),\\
\dot x &= r_1 (1-\epsilon_1^k \phi_+(r_1,\epsilon_1))(1+f(x,r_1),\\
 \dot \epsilon_1 &=\epsilon_1 F(r_1,x,\epsilon_1),
\end{align*}
after dividing the right hand side by $\epsilon_1$ (as promised in our description of the blowup approach, see \secref{mainResults}). In these equations we have introduced the following function
where
\begin{align*}
 F(r_1,x,\epsilon_1) =- (1-\epsilon_1^k\phi_+(r_1,\epsilon_1))(2x+r_1 g(x,r_1))-\epsilon_1^k \phi_+(r_1,\epsilon_1).
\end{align*}
\begin{remark}\remlab{yZp}
\qss{
Notice that within the invariant subset defined by $\epsilon_1=0$, we have $F_1(r_1,x,0)=-(2x+r_1g(x,r_1))$ and hence
\begin{align*}
\dot x &=y(1+f(x,y)),\\
 \dot y &= -y (2x+yg(x,y)),
\end{align*}
using that $r_1=y$. But this is just $yZ_+$, see \eqref{ZpmNF}. }
\end{remark}

Now, focus first on $x\in J$. Then for $r_1\ge 0$ and $\epsilon_1\ge 0$ but sufficiently small we have $F>0$ and hence the system is topological equivalent with the following version
\begin{align}
 \dot r_1 &=-r_1 ,\eqlab{r1xepsilon1}\\
\dot x &= r_1 (1-\epsilon_1^k \phi_+(r_1,\epsilon_1))\frac{(1+f(x,r_1)}{F(r_1,x,\epsilon_1)},\nonumber\\
 \dot \epsilon_1 &=\epsilon_1,\nonumber
\end{align}
The set $r_1=\epsilon_1=0$ is therefore a line of equilibria having stable and unstable manifolds contained within $\epsilon_1=0$ and $r_1=0$, respectively. We can straighten out the individual unstable manifolds of points on $x\in J,\,r_1=\epsilon_1=0$ by a transformation of the form $\tilde x\mapsto x=m(\tilde x,r_1)$, $r_1\in [0,\xi]$ with $m$ smooth, with $m'_{\tilde x}>0$. Applying this transformation to \eqref{r1xepsilon1} gives 
\begin{align*}
 \dot r_1 &=-r_1,\\
 \dot x &= \epsilon G(r_1,x,\epsilon_1),\\
 \dot \epsilon_1&=\epsilon_1,
\end{align*}
dropping the tilde on $x$. Recall here that $\epsilon=r_1\epsilon_1$. Therefore if we consider an initial condition with $r_1(0)=\delta>0$ small, then 
\begin{align}
 r_1(t) &=e^{-t}\delta,\eqlab{sol}\\
 x(t)&=x(0)+\mathcal O(\epsilon t),\nonumber\\
 \epsilon_1(t) &= e^{t}\epsilon_1(0).\nonumber
\end{align}

Now, we wish to extend the stable foliation of $S_\epsilon$ by the backward flow. For this let $x\in J$, after possibly decreasing $\nu>0$ and $\xi$, and consider the leaf $\mathcal F_{x,\epsilon}$ of the Fenichel foliation of $S_\epsilon$. We therefore flow this set forward $t=\mathcal O(\log \epsilon^{-1})$, which is the time it takes for $r_1$ to go from $\mathcal O(1)$ to $\mathcal O(\epsilon)$. This gives a new $x$, $x'=x\cdot t$, say, and a new leaf $\mathcal F_{x',\epsilon}$. Notice $\phi_t \left(\mathcal F_{x,\epsilon}\right)\subset \mathcal F_{x',\epsilon}$ and hence we extend $\mathcal F_{x,\epsilon}$ by flowing $\mathcal F_{x',\epsilon}$ backwards by time $t$. (In general, $\mathcal F_{x',\epsilon}$ will not be fully covered by the chart $\bar y=1$ and therefore we will have to work in separate charts.) We do this by using \eqref{sol}, which produces the extended leafs as images of Lipschitz mappings.  


\bibliography{refs}
\bibliographystyle{plain}
 \end{document}